\documentclass[final,onefignum,onetabnum]{siamart171218}

\usepackage{enumitem,kantlipsum}
\usepackage{subeqnarray}
\usepackage{stmaryrd}
\usepackage{lipsum}
\usepackage{amssymb}
\usepackage{amsfonts}
\usepackage{graphicx}
\usepackage{epstopdf}
\usepackage{algorithmic}
\ifpdf
  \DeclareGraphicsExtensions{.eps,.pdf,.png,.jpg}
\else
  \DeclareGraphicsExtensions{.eps}
\fi

\newsiamremark{remark}{Remark}
\newsiamremark{hypothesis}{Hypothesis}
\crefname{hypothesis}{Hypothesis}{Hypotheses}
\newsiamthm{claim}{Claim}
\newcommand{\Th}{\mathcal{T}_h}

\newcommand{\jump}[1]{\llbracket #1 \rrbracket}

\newcommand{\av}[1]{\{\!\!\{ #1 \}\!\!\}}

\newcommand{\enn}{\mathrm{ENN}} 
\newcommand{\bn}{{\bf{n}}}
\newcommand{\bu}{{\bf{u}}} 
\newcommand{\bU}{{\bf{U}}}

\headers{Enriched linearized Neural network approximation}{Seungil Kim, Gwanghyun Jo, and Young Ju Lee}

\title{A Locally Conservative Enriched Linearized Neural Network Approximation to Elliptic PDEs\thanks{Submitted to the editors DATE.
\funding{The work of S. Kim was  supported in part, by the National Research Foundation of Korea(NRF) grant funded by the Korea govenment(MSIT) (RS-2026-25591545). The work of Y. Lee was in part funded by NSF-DMS 2208499}}}

\author{Seungil Kim  \thanks{Kyung Hee University, Seoul, Korea
  (\email{sikim@khu.ac.kr}, \url{https://sikim.khu.ac.kr/}).}
\and Gwanghyun Jo \thanks{Department of Mathematics, Hanyang University, Erica, Ansan, Korea 
(\email{gwanghyun@hanyang.ac.kr},
\url{https://sites.google.com/view/gwanghyunjo}).}
\and Young Ju Lee \thanks{Mathematics Department, Texas State University, San Marcos, TX, US (\email{yjlee@txstate.edu})}}

\usepackage{amsopn}
\usepackage{comment} 

\newcommand{\Reals}[1]{{\rm I\! R}^{#1}}

\makeatletter
\newcommand*{\addFileDependency}[1]{
  \typeout{(#1)}
  \@addtofilelist{#1}
  \IfFileExists{#1}{}{\typeout{No file #1.}}
}
\makeatother

\makeatletter
\AtBeginDocument{%
  \def\refstepcounter@optarg[#1]#2{%
    \cref@old@refstepcounter{#2}%
    \cref@constructprefix{#2}{\cref@result}%
    \@ifundefined{cref@#1@alias}%
      {\def\@tempa{#1}}%
      {\def\@tempa{\csname cref@#1@alias\endcsname}}%
    \protected@edef\cref@currentlabel{%
      [\@tempa][\arabic{#2}][\cref@result]%
      \csname p@#2\endcsname\csname the#2\endcsname}}%
}
\makeatother

\providecommand{\Lkn}{L^k_n}
\providecommand{\Vnh}{V_{n,h}^{\enn}}
\providecommand{\Mzero}{\mathcal{M}^0(\mathcal{T}_h)}
\providecommand{\EGn}{{\rm EG}}
\providecommand{\avint}{\mathop{\,\rlap{--}\!\!\int}\nolimits}
\usepackage{booktabs}

\ifpdf
\hypersetup{
  pdftitle={A Locally Conservative Enriched Linearized Neural Network Approximation to Elliptic PDEs},
  pdfauthor={Young Ju Lee}
}
\fi

\begin{document}

\maketitle

\begin{abstract}
This paper presents a locally conservative Enriched Linearized Neural Network (ENN) method for the Darcy flow model. A linearized shallow ReLU$^k$ network, whose hidden-layer parameters are fixed on a quasi-uniform set of the sphere, carries the approximation, and piecewise constant functions on an auxiliary subdivision of the domain are added to ensure local mass conservation. The subdivision may consist of general, possibly curved, elements. Since neural network functions do not satisfy the inverse inequality, the stability and error analysis of the standard discontinuous Galerkin interior penalty formulation do not apply. 
The missing inverse inequality is remedied by an edge identity combined with a Galerkin least-squares (GLS) term, together with a single neural approximant that is optimal in the $L^2$, $H^1$ and $H^2$ norms simultaneously. We use the nonsymmetric interior penalty (NIPG) form, which is coercive for every positive penalty parameter. With $n$, the width of the shallow neural network, used in ENN, we prove an error estimate in the EG norm that is of the optimal $H^1$ order, i.e., $O(n^{-(r-1)/d})$ for solutions in $H^r$, $r\ge2$, and $O(n^{-(s-1)/d})$ for solutions in $H^s$, $\frac32<s\le2$, for any shape-regular subdivision, when the GLS parameter satisfies $\min\{h,n^{-1/d}\}^2\lesssim\tau\lesssim n^{-2/d}$ and the edge terms are scaled with $\min\{h_e,n^{-1/d}\}$, i.e., by the resolution of the network rather than of the mesh. A numerical Darcy flux is constructed, which is locally conservative, regardless of the size or the number of the elements. Sample numerical results demonstrate the correctness of the theory. In particular, a fixed $4\times4$ subdivision gives the same accuracy as a subdivision refined with the network, and the mass loss stays at round-off unlike a non-conservative neural network method.
\end{abstract}



\section{Introduction}

Neural networks are attractive for solving partial differential equations (PDEs) because a network $u_\theta:\Reals{d}\to\Reals{}$ is a global, mesh-free function that can be evaluated anywhere, its derivatives are available exactly, and its approximation power degrades with the dimension $d$ far more slowly than that of grid-based methods, for which a tensor-product grid with $M$ points per direction has $M^d$ points. These properties underlie the deep Ritz method \cite{e2018deep} and the finite neuron method \cite{xu2020finite}. The main practical difficulties are the nonconvexity and ill-conditioning of training; linearized neural networks, random feature methods and extreme learning machines reduce them by fixing the hidden-layer parameters, so that the discrete problem becomes linear. These methods, however, are not locally conservative. Conservation can be imposed as a hard constraint \cite{jagtap2020conservative,yan2023cell}; here we obtain it from the approximation space instead.

We propose the Enriched Linearized Neural Network (ENN) space $V_{n,h}^{\enn}$, following the enriched Galerkin (EG) method \cite{sunliu2009,lee2016locally}. Its principal component is the linearized ReLU$^k$ network space $L^k_n$ of width $n$ \cite{liu2025integral} (see \S\ref{LNN}), whose inner parameters $\{\theta_j\}_{j=1}^n$ are fixed on a quasi-uniform point set of the sphere. This component requires no mesh, its approximation rates are robust with respect to $d$ \cite{liu2025integral,siegel2024sharp}, it can be refined by adding neurons where the residual is large \cite{lu2026r}, and curved or implicitly defined domains are handled by evaluating integrals on the geometry directly. The second component consists of piecewise constants on an auxiliary partition $\Th$ of $\Omega$, whose elements may be general and possibly curved; its purpose is element-wise mass conservation, which a purely meshless expansion cannot guarantee. The method is thus a hybrid. Namely, the network carries the approximation, and a minimal low-order mesh carries the conservation.

The classical EG analysis \cite{lee2016locally} does not carry over directly to that of ENN, because neural network functions do not satisfy the inverse (discrete trace) inequalities on which the analysis of interior penalty methods rests \cite{wheeler1978elliptic,arnold2002unified,CiarletP-1978aa}. For Dirichlet problems on the purely neural space, Liu \cite{liu2026neural} showed that global $H^2$-conformity allows Green's identity to replace the inverse inequality in an unsymmetric Nitsche method with Galerkin least-squares (GLS) stabilization \cite{hughes1987new}, with the rate $n^{-(r-2)/d}$ in an $H^2$-type energy norm; local conservation is not addressed there. Our contributions are as follows.
\begin{enumerate}[leftmargin=0.7cm]
\item[(a)] We use the NIPG form on $V_{n,h}^{\enn}$, which is coercive for every positive penalty parameter, so that no inverse inequality is needed for stability (\S\ref{sec:coer}).
\item[(b)] We prove a C\'ea-type estimate whose constant, $1+\sqrt{10}$, is independent of $h$, $n$, $\tau$ and the penalty parameters, in contrast to the $O(\tau^{-1/2})$ constant in \cite{liu2026neural}; this allows $\tau$ to be scaled with $n$. With $\min\{h,n^{-1/d}\}^2\lesssim\tau\lesssim n^{-2/d}$ and the edge terms scaled with $\min\{h_e,n^{-1/d}\}$, it gives the $H^1$-optimal rate $O(n^{-(r-1)/d})$ in the EG norm for every shape-regular subdivision, coarse or fine, and $O(n^{-(s-1)/d})$ for solutions in $H^s$, $\frac32<s\le2$ (\S\ref{sec:error}). Hence the subdivision does not affect the accuracy, and it can be chosen solely according to the control volumes on which local mass conservation is required, e.g., the cells of a transport solver.
\item[(c)] The estimate uses a single approximant that is optimal in $L^2$, $H^1$ and $H^2$ simultaneously. This is implicit in the construction of \cite{liu2025integral}; we state it in Proposition~\ref{prop:simul} and prove it in Appendix~\ref{app:simul} for completeness. 
\item[(d)] We construct a locally conservative flux on any subdivision (\S\ref{sec3}) and compare it numerically with a non-conservative neural method (\S\ref{numerical_experiment}).
\end{enumerate}
We use the standard Sobolev spaces $H^k(\Omega)$, $H^0=L^2$, and $H({\rm div})$, denote by $\nabla_h$ the element-wise gradient, and write $a\lesssim b$ if $a\le Cb$ with $C$ independent of $n$, $h$ and the functions involved \cite{xu1992iterative}. Section~\ref{model} presents the model problem and the ENN space, \S\ref{scheme} the scheme and its analysis, \S\ref{sec3} the locally conservative flux, \S\ref{numerical_experiment} numerical experiments, and \S\ref{conclusion} concluding remarks.
\section{Model Problems}\label{model} 
Let $\Omega \subset \Reals{d}$ be a 
bounded Lipschitz domain, which need not be polygonal/polyhedral, with boundary 
$\partial\Omega = \overline{\Gamma_D} \cup \overline{\Gamma_N}$, where $\Gamma_D$ or $\Gamma_N$ may be empty. We shall consider to solve the following diffusion problem with mixed boundary conditions: 
\begin{subeqnarray}\label{eq:strong}
-\nabla\cdot(\kappa \nabla p) &=& f, \quad \text{ in } \Omega,\\
p &=& g_D, \quad \text{ on } \Gamma_D,\\
- \kappa \nabla p \cdot n &=& g_N, \quad \text{ on } \Gamma_N,
\end{subeqnarray}
where $\bu = -\kappa \nabla p$ is the Darcy flux and the conductivity $\kappa=\kappa(x)$ satisfies
\begin{equation}\label{eq:kappa}
\kappa\in W^{1,\infty}(\Omega),\qquad 0<\kappa_{\rm min}\le\kappa(x)\le\kappa_{\rm max}\quad\text{for a.e. }x\in\Omega .
\end{equation}
In particular, $\kappa$ is Lipschitz continuous on $\overline\Omega$, and $\kappa$ is \emph{not} assumed to be piecewise constant on the subdivision introduced below. We shall use the following two consequences of \eqref{eq:kappa}: for every $v\in H^1(D)$ and $w\in H^2(D)$ on a subdomain $D\subset\Omega$,
\begin{equation}\label{eq:kappa-bd}
\kappa_{\rm min}\|\nabla v\|_{0,D}^2\le\|\kappa^{1/2}\nabla v\|_{0,D}^2\le\kappa_{\rm max}\|\nabla v\|_{0,D}^2,
\,\,\,
\|\nabla\cdot\kappa\nabla w\|_{0,D}\le C_\kappa\|w\|_{H^2(D)},
\end{equation}
where $C_\kappa:=\kappa_{\rm max}+\|\nabla\kappa\|_{L^\infty(\Omega)}$. The hidden constants in our estimates depend on $\kappa_{\rm min}$, $\kappa_{\rm max}$ and $\|\nabla\kappa\|_{L^\infty(\Omega)}$. Thus, the result presented in this paper is not necessarily robust with respect to high contrast.  

Throughout the paper, we assume $f\in L^2(\Omega)$ for two technical reasons. First, the local conservation requires that $(f,\psi_T) = \int_T f\, dx$ makes sense for the characteristic function $\psi_T$ of an element $T$, which is not the case for a general $f \in H^{-1}(\Omega)$. Second, the GLS term introduced in \S\ref{scheme} requires $f \in L^2(\Omega)$. Consequently, the natural space for the solution $p$ is
\begin{equation}\label{eq:X} 
X = \{v \in H^1(\Omega) : \nabla \cdot \kappa \nabla v \in L^2(\Omega)\}. 
\end{equation} 
By \eqref{eq:kappa-bd}, $H^2(\Omega)\subset X$. 
The approximation space for $p$ consists of a linearized neural network and piecewise constant functions on a mesh of $\Omega$ \cite{lu2026invariant}, and the error is presented in terms of the number of neurons, $n$, and the mesh size, $h$.

\subsection{Linearized Neural Network Functions}
\label{LNN} 
Our discussion follows \cite{mao2026solving}. Let $\sigma_k(t) := \max\{t,0\}^k$, $t\in\Reals{}$, be the ReLU$^k$ activation, and for $x \in \Reals{d}$ let $x_{\rm aug} := (x-x_c,1) \in \Reals{d+1}$ be the augmented variable, where $x_c\in\Reals{d}$ is a fixed center, the same for all $x$ and all neurons (e.g., $x_c=0$, or the centroid of $\Omega$).
We let $\mathbb{D}$ be a symmetric dictionary of ReLU$^k$ ridge functions given by
\begin{equation} 
\mathbb{D} = \left\{
        \pm \sigma_k(\theta \cdot x_{\rm aug})
        :
        \theta = (\omega,b) \in \mathbb{S}^d  \,\, \mbox{ with } \omega \in \Reals{d} \mbox{ and } b \in \Reals{}
    \right\},
\end{equation} 
where $\mathbb{S}^d$ denotes the unit sphere in $\Reals{d+1}$. Note that $\theta\cdot x_{\rm aug}=\omega\cdot(x-x_c)+b$, so that $\theta$ carries both the direction and the bias of a neuron. For $M>0$, the shallow neural network class is $\Sigma_{n,M}(\mathbb{D}) :=\{ \sum_{i=1}^n a_i d_i: a_i \in \Reals{},\ d_i \in \mathbb{D},\ \sum_{i=1}^n |a_i| \le M\}$; each neuron is a global, genuinely mesh-less function. The following theorem gives the optimal approximation rate of ReLU$^k$ shallow networks for functions in the corresponding variation space.
\begin{theorem}[{\cite{siegel2024sharp}}]
Let $k \ge m$ and let $u \in K_1(\mathbb{D})$, where $K_1(\mathbb{D})$ denotes the variation space
associated with the dictionary $\mathbb{D}$. Then, with $M\simeq\|u\|_{K_1(\mathbb{D})}$,
\begin{equation}
\inf_{u_n \in \Sigma_{n,M}(\mathbb{D})}\|u-u_n\|_{H^m(\Omega)}
\leq C \|u\|_{K_1(\mathbb{D})}
n^{-\frac12-\frac{2(k-m)+1}{2d}},
\end{equation} 
where $C$ is independent of $u$ and $n$.
\end{theorem}
Thus, ReLU$^k$ shallow networks achieve an approximation rate whose exponent does not deteriorate with the dimension, with no reference to a mesh.
Since the approximation space must be contained in $X$, the neurons need to belong at least to $H^2$. Thus, we assume $k \geq 2$ throughout the paper, so that each neuron lies in $C^1(\overline\Omega)\cap W^{2,\infty}(\Omega)$ \cite{schumaker2007spline}. 

The same rate can be achieved even when the hidden-layer parameters are fixed in advance \cite{mao2026solving,liu2025integral}. Let $\{\theta_j\}_{j=1}^n \subset \mathbb{S}^d$ be a pre-determined set of parameters on the sphere. We define the associated
basis functions by
\begin{equation}\label{eq:phij} 
\phi_j(x) := \sigma_k(\theta_j \cdot x_{\rm aug}) = \sigma_k(w_j \cdot (x-x_c) + b_j), \qquad j=1,\ldots,n.
\end{equation} 
The finite neuron space is $L_n^k := \operatorname{span}\{\phi_1,\ldots,\phi_n\}$, and, analogously to $\Sigma_{n,M}(\mathbb{D})$, $L_{n,M}^k := \{ \sum_{j} a_j \phi_j : ( n \sum_{j} a_j^2)^{1/2}\leq M\}$. The parameters $\{\theta_j\}_{j=1}^n$ are chosen quasi-uniform, i.e., their separation distance and their fill distance on $\mathbb{S}^d$ are both comparable to
\begin{equation}\label{eq:delta}
\delta:=n^{-1/d}. 
\end{equation}
We shall set $r_{\max}:=\frac{d+2k+1}{2}$. Then, the following result can be established: 
\begin{theorem}[{\cite{liu2025integral}}]\label{thm:lmx}
Let $k \ge m$ and $u \in H^{r_{\max}}(\Omega)$. Assume that the points $\{\theta_j\}_{j=1}^n \subset \mathbb{S}^d$ form a
quasi-uniform grid on $\mathbb{S}^d$. Then, with $M\simeq\|u\|_{H^{r_{\max}}(\Omega)}$,
\[
    \inf_{u_n \in L_{n,M}^k}
    \|u-u_n\|_{H^m(\Omega)}
    \lesssim
    n^{-\frac12-\frac{2(k-m)+1}{2d}}
    \|u\|_{H^{\frac{d+2k+1}{2}}(\Omega)}.
\]
More generally, for $m\le r\le r_{\max}$ and $u\in H^r(\Omega)$, we have that \begin{equation} 
\inf_{u_n\in \Lkn}\|u-u_n\|_{H^m(\Omega)}\lesssim\delta^{r-m}\|u\|_{H^r(\Omega)}.
\end{equation} 
\end{theorem}
Thus, optimizing the hidden parameters is not necessary to achieve the optimal rate. Theorem~\ref{thm:lmx} provides, for each $m$, a possibly different approximant. The error
analysis in \S\ref{sec:error} needs one approximant that is optimal in $L^2$, $H^1$ and $H^2$ at the same time. Indeed, the construction of \cite{liu2025integral} has this property and we provide the proof in
Appendix~\ref{app:simul} for completeness.

\begin{proposition}[Simultaneous Approximation]\label{prop:simul}
Let $k\ge2$, $2\le r\le r_{\max}$, and let $\{\theta_j\}$ be quasi-uniform. For every $v\in H^r(\Omega)$
there is a single $v_n\in\Lkn$ such that
\begin{equation}\label{eq:simul}
\|v-v_n\|_{H^s(\Omega)}\lesssim\delta^{r-s}\|v\|_{H^r(\Omega)},\qquad s=0,1,2 .
\end{equation}
\end{proposition}
For finite elements, simultaneous approximation follows from the inverse inequality. Namely, on a quasi-uniform mesh, the $L^2$ projection $Q_h$ onto a $C^1$ element of degree $k$ (e.g., Argyris) satisfies $|u-Q_hu|_{H^m}\lesssim|u-I_hu|_{H^m}+h^{-m}\|u-I_hu\|_0\lesssim h^{k+1-m}|u|_{H^{k+1}}$, $m=0,1,2$ \cite{CiarletP-1978aa}. Since $\Lkn$ has no inverse inequality, this argument is not immediately available. 
This simultaneous approximation result is critical in our paper to obtain the $H^1$-optimal error estimate (Theorem~\ref{thm:error}).
\subsection{Mesh and Piecewise Constant Functions} 

Conforming spaces, including neural networks, do not provide a conservative flux (see \S \ref{sec:compare}). The simplest remedy is to enrich the solution space by piecewise constant functions, so that conservation holds on each element of a mesh \cite{gong2025positivity}. Strongly (pointwise) conservative neural networks have been devised \cite{he2026divergence}, but they are expensive and not necessary for robust simulations, since local conservation is sufficient \cite{gong2025positivity} in many practical applications. 

We let $\Th$ be a shape-regular division of the domain $\Omega$ into elements $T$ of diameter $h_T\le h\le 1$. The elements may be very general, i.e., they could be polygons/polyhedra and also elements with curved edges along $\partial\Omega$. The only assumption we shall require is that each $T$ is star-shaped with respect to a ball of radius comparable to $h_T$, i.e., there is a ball $B_T\subset T$ of radius $\rho_T\ge c\,h_T$, with $c>0$ independent of $T$ and $h$, such that for every $x\in T$ and $y\in B_T$ the segment $\{tx+(1-t)y: t\in[0,1]\}$ lies in $T$ \cite[Chapter~4]{brenner2008mathematical}, that each edge $e\subset\partial T$ satisfies $|e|\simeq h_T^{d-1}$, and that the number of edges per element is uniformly bounded, so that the trace inequality \eqref{eq:fmtrace} and the approximation property of the $L^2$-projection onto constants hold. This division will be used to define the space $\mathcal{M}^0(\mathcal{T}_h) = {\rm span}\{\psi_T\}_{T \in \mathcal{T}_h}$, where $\psi_T$ is the characteristic function, which is one on $T$ and zero elsewhere on $\Omega$. 

Consequently, using two spaces $L_{n}^k$ and $\mathcal{M}^0(\mathcal{T}_h)$, we define the enriched linearized neural network space as follows:
\begin{equation}
V_{n,h}^{\enn} = L_{n}^k + \mathcal{M}^0(\mathcal{T}_h) = \left \{ P = \sum_{j=1}^n c_j\,\phi_j + \sum_{T\in\Th} d_T\,\psi_T \quad \forall c_j, d_T \in \Reals{} \right \}. 
\end{equation}
Throughout this paper, we shall write for every $v\in V_{n,h}^{\enn}$, $v=v_n+v_0$ with $v_n\in L^k_n$ and $v_0\in\mathcal{M}^0(\mathcal{T}_h)$. In this setting, it holds that on each element $\nabla v=\nabla v_n$. 

The representation of $P\in V_{n,h}^{\enn}$ is not unique, even though the discrete solution is unique, which we shall show later. This gives a singular linear system \cite{lee2016locally} in general. For instance, with $\theta = (0,\cdots,0,1)$, the north pole of $\mathbb{S}^d$, the neuron is the constant $1$; more generally, a neuron whose hyperplane $\{\theta_j\cdot x_{\rm aug}=0\}$ does not meet $\Omega$ is either zero or a polynomial of degree $k$ on $\Omega$. Thus there can be many instances where the linear system can be singular.  
\section{Variational Formulation based on $V_{n,h}^{\enn}$}\label{scheme}

In this section, we shall derive the variational formulation based on the approximation space $V_{n,h}^{\enn}$ to solve the equation \eqref{eq:strong}. First, for a given subdivision $\mathcal{T}_h$, we let $\mathcal{E}_h$ denote the set of edges, $\mathcal{E}_h^o$  denote the set of interior edges,   $\mathcal{E}_h^{D,\partial}$ and $\mathcal{E}_h^{N,\partial}$ denote the sets of Dirichlet and Neumann boundary edges, respectively. $e \in \mathcal{E}_h$ will be called the $(d-1)$-dimensional pieces of the element boundaries edges. They are actually faces when $d=3$, but, we shall simply call them as edges for consistency. We write $h_e$ for the diameter of the edge $e$ in $\mathcal{E}_h$. For an interior edge shared by two adjacent elements, denoted by $T^\pm$ with unit outward normal $\bn^\pm$ to the boundary of $\partial T^\pm$, we define the average and jump as follows: 
\begin{equation}
\av{v} = \frac{1}{2}
\left ( v^+ + v^- \right ), \quad 
\jump{v} 
= v^+ \bn^+ + v^- \bn^-, \quad \jump{\sigma}=\sigma^+\cdot\bn^++\sigma^-\cdot\bn^-,
\end{equation}
for scalar $v$ and vector $\sigma$. On boundary edges, we set $\av{\sigma} = \sigma$ and $\jump{v} = v\bn_{\Omega}$. Compared to the original DG or EG, the variational formulation is not much different at all. Only the conforming subspace changes. Multiplying the governing equation \eqref{eq:strong} by a test function $w \in V_{n,h}^{\enn}$ and integrating over an element $T\in\mathcal{T}_h$ gives
\begin{equation}
\int_T \kappa\nabla p\cdot\nabla w\,dx-\int_{\partial T}\kappa\nabla p\cdot \bn_{_T} \,w\,ds = \int_T fw\,dx.
\end{equation}
Summing over all elements, we get 
\begin{equation}\label{eq:sum} 
(\kappa\nabla p,\nabla w)_{\mathcal{T}_h}-\sum_{T\in\mathcal{T}_h}\int_{\partial T}\kappa\nabla p\cdot \bn_{_T} \,w\,ds=(f,w)_{\mathcal{T}_h}, 
\end{equation}
where $
(v,w)_{\mathcal{T}_h} := \sum_{T\in\mathcal{T}_h} \int_T v w\,dx.$ We shall introduce a couple of notation for the set of edges and discontinuous permeability coefficients as follows: 
\begin{equation}
\mathcal{E}_h^1 := \mathcal{E}_h^o \cup \mathcal{E}_h^{D,\partial}, \quad \beta_e = \frac{\kappa^-}{\kappa^+ + \kappa^-} \quad \mbox{ and } \quad 
\kappa_e = \frac{2\kappa^+\kappa^-}{\kappa^+ + \kappa^-}.
\end{equation}
Since $\kappa$ is continuous by \eqref{eq:kappa}, $\kappa^+=\kappa^-$ on every interior edge, so that $\beta_e=\frac12$ and $\kappa_e=\kappa|_e$, which is a function on $e$ (not a constant); on a Dirichlet edge, we set $\beta_e=1$ and $\kappa_e=\kappa|_e$. In particular, $\langle \hat h_e^{-1}\kappa_e\jump{v},\jump{w}\rangle_{\mathcal{E}_h^1}=\sum_{e\in\mathcal{E}_h^1}\hat h_e^{-1}(\kappa\jump{v},\jump{w})_{0,e}$.
The variational formulation is then given as follows: Find $P \in V_{n,h}^{\enn}$ such that 
\begin{equation}\label{eq:weak}
A_\xi(P,w) = F_\xi(w), \quad \forall w \in V_{n,h}^{\enn}, 
\end{equation}
where 
\begin{eqnarray*}
A_\xi(v,w) &:=& (\kappa\nabla v,\nabla w)_{\mathcal{T}_h} - \left\langle \av{\kappa\nabla v}_{\beta_e},\jump{w} \right \rangle_{\mathcal{E}_h^1} + \xi \left \langle \jump{v},\av{\kappa\nabla w}_{\beta_e}\right\rangle_{\mathcal{E}_h^1} \nonumber \\
&& + \alpha \left\langle \hat h_e^{-1} \kappa_e \jump{v},\jump{w} \right \rangle_{\mathcal{E}_h^1} + \tau (\nabla \cdot \kappa \nabla v, \nabla \cdot \kappa \nabla w)_{\mathcal{T}_h} \nonumber \\
&& + \gamma \left\langle \hat h_e\,\kappa\nabla v\cdot\bn,\kappa\nabla w\cdot\bn\right\rangle_{\mathcal{E}_h^{N,\partial}},
\end{eqnarray*}
where $\alpha$ and $\tau$ are the penalty parameters, $\gamma \geq 0$ is the Neumann least-squares parameter, $h_e$ is the diameter of the edge $e$, $\hat h_e:=\min\{h_e,\delta\}$ with $\delta$ given in \eqref{eq:delta}, $\xi$ determines the interior penalty variant,  
\begin{equation}
\av{\kappa\nabla v}_{\beta_e} =
\beta_e (\kappa\nabla v)^+ + 
(1-\beta_e)(\kappa\nabla v)^- \quad \mbox{ and } \quad \langle a,b\rangle_{\mathcal{E}_h^1}:=\sum_{e\in\mathcal{E}_h^1}\int_e a\cdot b\,ds. 
\end{equation}
Since $\kappa^+=\kappa^-=\kappa_e$ on each edge, it holds for every piecewise $H^{3/2+\epsilon}$ function $v$ that
\begin{equation}\label{eq:kappae} 
\av{\kappa\nabla v}_{\beta_e} =
\kappa_e \av{\nabla v}=\kappa\av{\nabla v}.
\end{equation} 
Moreover, if $\nabla v$ is continuous across $e$, in particular for $v\in V_{n,h}^{\enn}$ (since $\nabla v=\nabla v_n$ with $v_n\in C^1(\overline\Omega)$) and for $v\in H^2(\Omega)$, then $\av{\nabla v}=\nabla v$ on $e$. 
The corresponding right-hand side functional is given as follows: 
\begin{eqnarray*}
F_\xi (w) &=& (f,w)_{\mathcal{T}_h} - \langle g_N,w \rangle_{\mathcal{E}_h^{N,\partial}} + \xi \left\langle g_D, \av{\kappa\nabla w}_{\beta_e}\cdot\bn\right\rangle_{\mathcal{E}_h^{D,\partial}} + \alpha \left \langle \hat h_e^{-1}\kappa_e g_D,w \right \rangle_{\mathcal{E}_h^{D,\partial}} \\
&& - \tau (f, \nabla \cdot \kappa \nabla w)_{\mathcal{T}_h} - \gamma\left\langle \hat h_e\,g_N,\kappa\nabla w\cdot\bn\right\rangle_{\mathcal{E}_h^{N,\partial}}. 
\end{eqnarray*}
We note that $\xi$ determines the type of DG bilinear form, i.e., for $\xi = -1$, it is SIPG, for $\xi = 0$, it is IIPG while for $\xi = 1$, it is NIPG. In this paper, we use the NIPG form, $\xi=1$; the IIPG and SIPG forms are discussed in Remark~\ref{rem:ipvariants}. Since $\nabla\cdot\kappa\nabla p=-f$ and $\kappa\nabla p\cdot\bn=-g_N$ on $\Gamma_N$, the GLS and the Neumann least-squares contributions to $F_\xi$ carry a minus sign. The Neumann least-squares term is used only for mixed boundary conditions (see the remark after Theorem~\ref{thm:cea}); it plays no role if $\Gamma_N=\emptyset$, and it is not needed for the pure Neumann problem (Theorem~\ref{thm:neumann}). 
\begin{remark}
The standard DG or EG has been formulated without the Galerkin least-square stabilized term originally proposed in \cite{hughes1987new}. For the neural network enriched by the piecewise constants, this term is needed in the error analysis (see \S\ref{sec:error}), and it can also tackle the Dirichlet boundary condition as discussed in \cite{liu2026neural}. For the NIPG variant, the term is not needed for coercivity (Lemma~\ref{lem:coer}), and for a pure Neumann problem it is not needed at all (Theorem~\ref{thm:neumann}).  
\end{remark}
\subsection{Interpolation in the Enriched Linearized Neural Network Space}\label{sec:interp}
We first define a neural interpolation operator $I_n:H^r(\Omega)\mapsto\Lkn$, $2 \le r \le r_{\max}$, by
\begin{equation}\label{eq:In}
I_nv:=v_n,
\end{equation}
where $v_n$ is the simultaneous approximant of Proposition~\ref{prop:simul}. (In contrast to a
best approximation in a single norm, such as $\arg\min_{w_n}\|v-w_n\|_{H^2(\Omega)}$, $I_nv$ is
optimal in $H^1$ and $H^2$ at the same time.) Additionally, we let
\begin{equation}
Q_h^0:L^2(\Omega)\mapsto\Mzero
\end{equation}
be the $L^2$-projection onto the space $\Mzero$. We now define the neural enriched interpolation
operator $\Pi_{n,h}:H^r(\Omega)\mapsto\Vnh$ by
\begin{equation}\label{eq:Pi}
\Pi_{n,h}v:=I_nv+Q_h^0(v-I_nv).
\end{equation}
We note that, with
\begin{equation}
r_n:=v-I_nv,
\end{equation}
it holds that
\begin{equation}\label{eq:vPi}
v-\Pi_{n,h}v=r_n-Q_h^0r_n .
\end{equation}
We now define the EG norm by
\begin{equation}\label{eq:EG}
\begin{aligned}
\|w\|_{\EGn}^2:={}& \sum_{T\in\Th} \|\kappa^{1/2} \nabla w\|_{0,T}^2+\alpha\sum_{e\in\mathcal{E}_h^1}\hat h_e^{-1}\|\kappa^{1/2}  \jump w\|_{0,e}^2
+\tau \|\nabla\cdot\kappa\nabla w\|_{\mathcal{T}_h}^2\\
&+\gamma\sum_{e\in\mathcal{E}_h^{N,\partial}}\hat h_e\|\kappa\nabla w\cdot\bn\|_{0,e}^2,
\end{aligned}
\end{equation}
We also write $|w|_*^2 := \sum_{e\in\mathcal{E}_h^1}\hat h_e^{-1}\|\kappa^{1/2} \jump w\|_{0,e}^2$. 
It is a norm on $\Vnh$ when $|\Gamma_D|>0$. Namely, if $\|w\|_{\EGn}=0$, then $w$ is piecewise constant without
jumps and vanishes on $\Gamma_D$. If $\Gamma_D=\emptyset$, it is a norm on the subspace of functions with zero mean (Theorem~\ref{thm:neumann}).
\begin{remark}[Edge weights]\label{rem:weights}
The edge weights in \eqref{eq:EG} balance the edge terms against the element terms. For a finite element function $v_h$ of fixed degree on a shape-regular mesh, the inverse inequality $|v_h|_{H^2(T)}\lesssim h_T^{-1}|v_h|_{H^1(T)}$ \cite{CiarletP-1978aa} and the multiplicative trace inequality \eqref{eq:fmtrace} give the discrete trace inequality $h_e\|\nabla v_h\|_{0,e}^2\lesssim\|\nabla v_h\|_{0,T}^2$, and the Poincar\'e inequality gives $h_e^{-1}\|v-Q_h^0v\|_{0,e}^2\lesssim\|\nabla v\|_{0,T}^2$; this is why the weights are $h_e^{\pm1}$ in the standard EG and DG methods. Since $\Lkn$ satisfies no inverse inequality of this form, we use the weights $\hat h_e^{\pm1}$, $\hat h_e=\min\{h_e,\delta\}$, for which Lemma~\ref{lem:elem} below bounds the edge terms of the error by element terms with constants independent of $h$.
\end{remark}
 Since $Q_h^0r_n$ is constant on each element, we have
\begin{equation}
\nabla(v-\Pi_{n,h}v)=\nabla(r_n-Q_h^0r_n)=\nabla r_n\quad\text{on each }T\in\Th .
\end{equation}
Therefore,
\begin{equation}\label{eq:gradPi}
\sum_{T\in\Th}\|\kappa^{1/2} \nabla(v-\Pi_{n,h}v)\|_{0,T}^2=\sum_{T\in\Th} \|\kappa^{1/2} \nabla(v-I_nv)\|_{0,T}^2 .
\end{equation}
We use the multiplicative trace inequality \cite[Theorem~1.6.6]{brenner2008mathematical}, which reads as follows for shape-regular elements; for curved elements it follows in the same way, by scaling the trace theorem on a Lipschitz domain, provided the scaled elements have uniformly Lipschitz boundaries: 
\begin{equation}\label{eq:fmtrace}
\|\phi\|_{0,\partial T}^2\lesssim h_T^{-1}\|\phi\|_{0,T}^2+\|\phi\|_{0,T}|\phi|_{H^1(T)},\qquad\forall\phi\in H^1(T).
\end{equation}
In contrast to the additive form $\|\phi\|_{0,\partial T}^2\lesssim h_T^{-1}\|\phi\|_{0,T}^2+h_T|\phi|_{H^1(T)}^2$, its second term does not contain a positive power of $h_T$, which is essential on coarse elements, i.e., ($h_T > \delta$). By the assumptions on $\Th$ (shape regularity and $|e|\simeq h_T^{d-1}$), $h_e\simeq h_T$ and, thus  $\hat h_e\simeq\min\{h_T,\delta\}$ for $e\subset\partial T$.

\begin{lemma}\label{lem:elem}
Let $\rho\in H^1(\Omega)$ and $\eta:=\rho-Q_h^0\rho$. Then, for every $T\in\Th$ and every edge $e\subset\partial T$,
\begin{equation}\label{eq:etaedge}
\hat h_e^{-1}\|\eta\|_{0,e}^2\lesssim\|\nabla\rho\|_{0,T}^2+\delta^{-2}\|\rho\|_{0,T}^2 .
\end{equation}
If, in addition, $\rho|_T\in H^2(T)$, then
\begin{equation}\label{eq:edgegrad}
\hat h_e\|\nabla\rho\|_{0,e}^2\lesssim\|\nabla\rho\|_{0,T}^2+\delta\|\nabla\rho\|_{0,T}|\rho|_{H^2(T)} .
\end{equation}
\end{lemma}
\begin{proof}
Since $Q_h^0\rho|_T$ is the mean value of $\rho$ on $T$, we have $\nabla\eta=\nabla\rho$ on $T$, $\|\eta\|_{0,T}\le\|\rho\|_{0,T}$, and, by the Poincar\'e inequality \cite[Chapter~4]{brenner2008mathematical},
\begin{equation}\label{eq:poinc}
\|\eta\|_{0,T}\le C_Ph_T\|\nabla\rho\|_{0,T}.
\end{equation}
By \eqref{eq:fmtrace}, $\|\eta\|_{0,e}^2\lesssim h_T^{-1}\|\eta\|_{0,T}^2+\|\eta\|_{0,T}\|\nabla\rho\|_{0,T}$. If $h_T\le\delta$, then $\hat h_e\simeq h_T$, and \eqref{eq:poinc} gives $\hat h_e^{-1}\|\eta\|_{0,e}^2\lesssim\|\nabla\rho\|_{0,T}^2$. If $h_T>\delta$, then $\hat h_e\simeq\delta$, and $\|\eta\|_{0,T}\le\|\rho\|_{0,T}$ gives $\hat h_e^{-1}\|\eta\|_{0,e}^2\lesssim\delta^{-2}\|\rho\|_{0,T}^2+\delta^{-1}\|\rho\|_{0,T}\|\nabla\rho\|_{0,T}$, which proves \eqref{eq:etaedge}. Similarly, \eqref{eq:fmtrace} with $\phi=\partial_i\rho$ gives $\hat h_e\|\nabla\rho\|_{0,e}^2\lesssim(\hat h_e/h_T)\|\nabla\rho\|_{0,T}^2+\hat h_e\|\nabla\rho\|_{0,T}|\rho|_{H^2(T)}$, and $\hat h_e\lesssim\min\{h_T,\delta\}$ gives \eqref{eq:edgegrad}. This completes the proof.
\end{proof} 
For finite element spaces, where the resolution scale is the mesh size, i.e., $\delta\simeq h$, we have $\hat h_e\simeq h_e$, and Lemma~\ref{lem:elem} reduces to the standard estimates of the EG and DG analysis \cite{lee2016locally,arnold2002unified}.

For an interior edge $e$, since $v$ and $I_nv$ are continuous, $r_n$ is also continuous, and $\jump{v-\Pi_{n,h}v}=-\jump{Q_h^0r_n}=\jump{r_n-Q_h^0r_n}$; on a Dirichlet edge, $\jump{v-\Pi_{n,h}v}=(r_n-Q_h^0r_n)\bn$. Hence \eqref{eq:etaedge} with $\rho=r_n$ and \eqref{eq:kappa-bd} give
\begin{equation}\label{eq:jumpPi}
|v-\Pi_{n,h}v|_*^2\lesssim\kappa_{\rm max}\big(\|\nabla r_n\|_0^2+\delta^{-2}\|r_n\|_0^2\big).
\end{equation}
Similarly, since $\nabla(Q_h^0r_n)=0$ on each element, \eqref{eq:kappa-bd} gives $\|\nabla\cdot\kappa\nabla(v-\Pi_{n,h}v)\|_{0,T}=\|\nabla\cdot\kappa\nabla r_n\|_{0,T}\le C_\kappa\|r_n\|_{H^2(T)}$, and \eqref{eq:edgegrad} bounds the Neumann term. Therefore,
\begin{equation}\label{eq:PiEG}
\|v-\Pi_{n,h}v\|_{\EGn}^2\lesssim(1+\alpha+\gamma)\big(\|\nabla r_n\|_0^2+\delta^{-2}\|r_n\|_0^2\big)+(\tau+\gamma\delta^2)\|r_n\|_{H^2(\Omega)}^2 .
\end{equation}
Thus the interpolation error in the enriched neural Galerkin norm is controlled by the $L^2$, $H^1$ and
$H^2$ approximation errors of the \emph{same} linearized neural network approximant. By
Proposition~\ref{prop:simul}, for $v\in H^r(\Omega)$, $2\le r\le r_{\max}$,
\begin{equation}\label{eq:PiRate}
\|v-\Pi_{n,h}v\|_{\EGn}\lesssim\big(\delta^{r-1}+\tau^{1/2} \delta^{r-2}\big)\|v\|_{H^r(\Omega)} ,
\end{equation}
for every shape-regular subdivision, independently of $h$. In particular, for a fixed $\tau>0$ and $r=r_{\max}$, the bound is of $H^2$ type,
\begin{equation}\label{eq:PiRateFixed}
\|v-\Pi_{n,h}v\|_{\EGn}\lesssim n^{-\frac12-\frac{2(k-2)+1}{2d}}\|v\|_{H^{r_{\max}}(\Omega)},
\end{equation}
while for $\tau\lesssim\delta^2$ it is of optimal $H^1$ order. The error estimate (Theorem~\ref{thm:error}) also needs a lower bound on $\tau$ to control $\tau^{-1/2}\|\eta\|_0$; since $\|\eta\|_0$ is bounded both by the $L^2$ error and, through the Poincar\'e inequality, by $h$ times the $H^1$ error, the lower bound $\min\{h,\delta\}^2$ suffices. Therefore, we shall assume that
\begin{equation}\label{cond} 
\min\{h,\delta\}^2\lesssim\tau\lesssim\delta^2=n^{-2/d}, 
\end{equation}
with no condition on the mesh size $h$. In particular, $\tau\simeq\delta^2$ is admissible for every subdivision, and for $h\lesssim\delta$ any $h^2\lesssim\tau\lesssim\delta^2$ is admissible as well. Condition \eqref{cond} is sufficient to obtain the optimal $H^1$ error rate when $|\Gamma_D| > 0$.For $|\Gamma_D| = 0$, we remark that any $\tau\geq 0$ can be used (Theorem~\ref{thm:neumann}). 


\subsection{Continuity and Coercivity of the bilinear form $A_\xi$}\label{sec:coer}
Since the conforming part is no longer a polynomial finite element space, the usual finite
element inverse inequality $h_e\|\nabla v\|_{0,e}^2\lesssim\|\nabla v\|_{0,T}^2$ \cite{CiarletP-1978aa} cannot be
applied. The NIPG form does not need it for stability; the following constant replaces it in the continuity estimate \eqref{eq:cont} and in the flux estimate of \S\ref{sec3}.

\begin{lemma}\label{lem:ctr}
Let 
\begin{equation}\label{eq:ctr}
C_{tr}(n,h):=\sup_{v\in\Lkn,\ \nabla v\ne0}
\frac{\sum_{e\in\mathcal{E}_h^1}\hat h_e\|\kappa^{1/2} \av{\nabla v}\|_{0,e}^2}{\|\kappa^{1/2}\nabla v\|_0^2+\tau\|\nabla\cdot\kappa\nabla v\|_0^2},
\end{equation}
and, if $\Gamma_N\ne\emptyset$, let the Neumann trace constant be
\begin{equation}\label{eq:cn}
C_{N}(n,h):=\sup_{v\in\Lkn,\ \nabla v\ne0}
\frac{\sum_{e\in\mathcal{E}_h^{N,\partial}}\hat h_e\|\kappa\nabla v\cdot\bn\|_{0,e}^2}{\|\kappa^{1/2}\nabla v\|_0^2+\tau\|\nabla\cdot\kappa\nabla v\|_0^2}.
\end{equation}
Then $C_{tr}(n,h)<\infty$ and $C_N(n,h)<\infty$ for $\tau \geq 0$.  
\end{lemma}
\begin{proof} 
By definition, it holds that 
\begin{equation}\label{eq:ctrineq}
\sum_{e\in\mathcal{E}_h^1}\hat h_e\|\kappa^{1/2} \av{\nabla v}\|_{0,e}^2\le C_{tr}(n,h)\Big(\|\kappa^{1/2}\nabla v\|_0^2
+\tau\|\nabla\cdot\kappa\nabla v\|_0^2\Big),\quad\forall v\in\Vnh .
\end{equation}
The constant $C_{tr}(n,h)$ can be understood as the largest eigenvalue of a generalized eigenvalue problem on $\Lkn$. The numerator and the denominator are quadratic forms on the finite dimensional quotient space
of $\Lkn$ modulo constants, and the denominator is positive definite there, since it is bounded below by $\kappa_{\rm min}\|\nabla v\|_0^2$; hence the supremum is the largest
generalized eigenvalue, which is finite. Both sides only depend on the neural part $v_n$ of $v=v_n+v_0$. The same argument applies to $C_N(n,h)$. This completes the proof. 
\end{proof}

Unlike the finite element case, $C_{tr}(n,h)$ is not known to be bounded uniformly in $n$, since no inverse inequality for $L^k_n$ with a constant independent of $n$ is available; see the remark after Theorem~\ref{thm:fluxerror}.
\begin{lemma}\label{lem:coer}
For $\xi=1$, i.e., NIPG, and all $\alpha,\tau,\gamma\ge0$,
\begin{equation}\label{eq:nipg}
A_1(v,v)=\|v\|_{\EGn}^2,\qquad\forall v\in\Vnh .
\end{equation}
\end{lemma}
\begin{proof}
For any $\xi$, the following identity holds:
\begin{equation}\label{eq:Avv}
A_\xi(v,v) = E(v)-(1-\xi)\langle\av{\kappa\nabla v}_{\beta_e},\jump v\rangle_{\mathcal{E}_h^1}+\alpha|v|_*^2 ,
\end{equation}
where $E(v):=\|\kappa^{1/2}\nabla v\|_{\Th}^2+\tau\|\nabla\cdot\kappa\nabla v\|_{\Th}^2+\gamma\sum_{e\in\mathcal{E}_h^{N,\partial}}\hat h_e\|\kappa\nabla v\cdot\bn\|_{0,e}^2$, so that $\|v\|_{\EGn}^2=E(v)+\alpha|v|_*^2$. For $\xi=1$, the cross term vanishes. This completes the proof.
\end{proof}
\begin{remark}[IIPG and SIPG]\label{rem:ipvariants}
For $\xi\in\{0,-1\}$, the Cauchy-Schwarz inequality, \eqref{eq:kappae} and Lemma~\ref{lem:ctr} give $|\langle\av{\kappa\nabla v}_{\beta_e},\jump v\rangle_{\mathcal{E}_h^1}|\le C_{tr}^{1/2}E(v)^{1/2}|v|_*$, and hence $A_\xi(v,v)\ge\frac12\|v\|_{\EGn}^2$ for $\alpha\ge(1-\xi)^2C_{tr}(n,h)$. The analysis below then carries over, but with constants depending on $\alpha\ge(1-\xi)^2C_{tr}(n,h)$, which is not known to be bounded uniformly in $n$. We therefore restrict the analysis to NIPG; the local conservation of \S\ref{sec3} holds for all three variants.
\end{remark}
For NIPG the coercivity is independent of the mesh size $h$
and of its quality, which is similar for classical NIPG \cite{riviere1999improved,riviere2001priori}. From  \eqref{eq:ctrineq}, we can also obtain the following continuity result: for $v,w\in\Vnh$,
\begin{equation}\label{eq:cont}
|A_1(v,w)|\le\Big(1+2\big(C_{tr}(n,h)/\alpha\big)^{1/2}\Big)\|v\|_{\EGn}\|w\|_{\EGn},\qquad\forall v,w\in\Vnh .
\end{equation}
The continuity constant may grow with $n$ through $C_{tr}$. The key idea is to remove this constant with  replacing $\|\cdot\|_{\EGn}$ by stronger norm. The error estimate follows by analyzing the newly introduced norm. 

\subsection{Consistency of the weak formulation \eqref{eq:weak} and its error analysis}\label{sec:error}
Assume that the exact solution $p$ is sufficiently regular, $p|_T\in H^{3/2+\epsilon}(T)$ for all
$T\in\Th$ and some $\epsilon>0$, and satisfies the interface conditions
\begin{equation}\label{note1}
\jump p=0\qquad\text{and}\qquad\jump{\kappa\nabla p}=0\qquad\text{on interior edges}.
\end{equation}
%
We present the consistency result, which is also standard in DG or EG. 
\begin{theorem}\label{thm:consistency}
Under the above regularity assumptions on $p$, $A_\xi(p,w) = F_\xi(w)$ for all $w\in\Vnh$, $\xi\in\{-1,0,1\}$ and 
$\alpha,\tau\ge0$. Therefore,
\begin{equation}\label{eq:galorth}
A_\xi(p-P,w)=0,\qquad\forall w\in\Vnh .
\end{equation}
\end{theorem}
We are now in a position to establish the error estimate. A C\'ea-type argument is based on
coercivity, consistency and the continuity of $A_\xi(\eta,\chi)$ for $\eta\notin\Vnh$ and $\chi\in\Vnh$. The latter cannot be obtained in
$\|\cdot\|_{\EGn}$ alone, since the edge terms of $A_\xi$ involve the traces of $\nabla\chi_n$, which require either an inverse inequality or the GLS term. We shall introduce an instrumental identity, so-called the edge identity \cite{arnold2002unified}:
\begin{lemma}\label{lem:edge}
Let $\sigma\in H(\mathrm{div};\Omega)$ have $L^2$ normal traces on all edges, and let $\eta$ be
piecewise $H^1$. Then
\begin{equation}\label{eq:faceid}
\sum_{e \in\mathcal{E}_h}\int_e\sigma\cdot\jump\eta\,ds=(\sigma,\nabla \eta)_{\mathcal{T}_h} +(\nabla\cdot\sigma,\eta). 
\end{equation}
\end{lemma}
For $\chi\in\Vnh$, the field $\kappa\nabla\chi_n$ belongs to $H(\mathrm{div};\Omega)$ since
$\kappa\in W^{1,\infty}(\Omega)$ and $\chi_n\in C^1(\overline\Omega)\cap W^{2,\infty}(\Omega)$; moreover
$\av{\kappa\nabla\chi}_{\beta_e}=\kappa\nabla\chi_n$ on every edge, and the GLS term gives
$\|\nabla\cdot\kappa\nabla\chi_n\|_0\le\tau^{-1/2}\|\chi\|_{\EGn}$. For functions $\eta$ with $\eta|_T\in H^{3/2+\epsilon}(T)$ and $\nabla\cdot\kappa\nabla\eta|_T\in L^2(T)$ for all $T\in\Th$, so that $\nabla\eta$ has $L^2$ traces on all edges, we define the augmented norm
\begin{eqnarray}\label{eq:augnorm}
|||\eta|||^2 &:=& \|\eta\|_{\EGn}^2+\alpha^{-1}\sum_{e\in\mathcal{E}_h^1}\hat h_e\|\kappa^{1/2}\av{\nabla\eta}\|_{0,e}^2 \\
&& +\tau^{-1}\|\eta\|_0^2+\gamma^{-1}\sum_{e\in\mathcal{E}_h^{N,\partial}}\hat h_e^{-1}\|\eta\|_{0,e}^2 ,\nonumber 
\end{eqnarray}
so that $\|\eta\|_{\EGn}\le|||\eta|||$; the last term is absent if $\Gamma_N=\emptyset$. We shall establish that this norm allows us to remove $\alpha$ and $C_{\rm tr}$ in the continuity-type estimate. 

\begin{theorem}\label{thm:cea}
Let $\Gamma_D\ne\emptyset$, and let $\gamma>0$ if $\Gamma_N\ne\emptyset$. Let $p\in X$ satisfy $p|_T\in H^{3/2+\epsilon}(T)$ for all $T\in\Th$ and some $\epsilon>0$ and \eqref{note1}, and let $\xi=1$ and $\alpha>0$. Then
\begin{equation}\label{eq:cea}
\|p-P\|_{\EGn}\le\big(1+\sqrt{10}\big)\inf_{z\in\Vnh}|||p-z|||,
\end{equation}
where the constant $1+\sqrt{10}$ is independent of $n$, $h$, $\tau$, $\alpha$, $\gamma$ and $\kappa$.
\end{theorem}
\begin{proof}
Let $z\in\Vnh$, $\eta:=p-z$ and $\chi:=z-P\in\Vnh$. Since $z_n\in W^{2,\infty}(\Omega)$ and $\nabla\cdot\kappa\nabla p=-f\in L^2(\Omega)$, $|||\eta|||$ is finite. By Theorem~\ref{thm:consistency} and Lemma~\ref{lem:coer},
\begin{equation}\label{eq:cea1}
\|\chi\|_{\EGn}^2= A_1(\chi,\chi)=A_1(\chi,\chi)-A_1(p-P,\chi)=-A_1(\eta,\chi).
\end{equation}
We bound the six terms of $A_1(\eta,\chi)$ by the Cauchy-Schwarz inequality. Let $a_1:=\|\kappa^{1/2}\nabla_h\eta\|_0$, $a_2:=\alpha^{1/2}|\eta|_*$, $a_3:=\tau^{1/2}\|\nabla\cdot\kappa\nabla\eta\|_{\Th}$, $a_4^2:=\alpha^{-1}\sum_e\hat h_e\|\kappa^{1/2}\av{\nabla\eta}\|_{0,e}^2$, $a_5:=\tau^{-1/2}\|\eta\|_0$, $a_6^2:=\gamma\sum_{e\in\mathcal{E}_h^{N,\partial}}\hat h_e\|\kappa\nabla\eta\cdot\bn\|_{0,e}^2$ and $a_7^2:=\gamma^{-1}\sum_{e\in\mathcal{E}_h^{N,\partial}}\hat h_e^{-1}\|\eta\|_{0,e}^2$, so that $\sum_{i=1}^7a_i^2=|||\eta|||^2$. The volume, penalty, GLS and Neumann least-squares terms are bounded by $a_1$, $a_2$, $a_3$ and $a_6$ times $\|\chi\|_{\EGn}$, respectively. By \eqref{eq:kappae} and the Cauchy-Schwarz inequality with the weights $\hat h_e^{\pm1/2}\kappa^{1/2}$, $|\langle\av{\kappa\nabla\eta}_{\beta_e},\jump\chi\rangle_{\mathcal{E}_h^1}|\le a_4\alpha^{1/2}|\chi|_*\le a_4\|\chi\|_{\EGn}$. Finally, since $\av{\kappa\nabla\chi}_{\beta_e}=\kappa\nabla\chi_n$ and $\mathcal{E}_h^1=\mathcal{E}_h\setminus\mathcal{E}_h^{N,\partial}$, Lemma~\ref{lem:edge} with $\sigma=\kappa\nabla\chi_n\in H(\mathrm{div};\Omega)$, which, as in \cite{liu2026neural} for the boundary terms, lets the GLS term replace the inverse inequality, gives
\begin{equation}\label{eq:xiN}
\langle\jump\eta,\av{\kappa\nabla\chi}_{\beta_e}\rangle_{\mathcal{E}_h^1}=(\kappa\nabla\chi_n,\nabla_h\eta)_{\Th}+(\nabla\cdot\kappa\nabla\chi_n,\eta)-\langle\eta,\kappa\nabla\chi_n\cdot\bn\rangle_{\mathcal{E}_h^{N,\partial}} .
\end{equation}
The first two terms are bounded by $(a_1+a_5)\|\chi\|_{\EGn}$, using the GLS part of $\|\chi\|_{\EGn}$, and the last one by $a_7\|\chi\|_{\EGn}$, using the Cauchy-Schwarz inequality with the weights $\hat h_e^{\mp1/2}$ and the Neumann least-squares part of $\|\chi\|_{\EGn}$.
Since $\xi=1$, the Cauchy-Schwarz inequality in $\Reals{7}$ gives
\[
|A_1(\eta,\chi)|\le\big(2a_1+a_2+a_3+a_4+a_5+a_6+a_7\big)\|\chi\|_{\EGn}\le \sqrt{10}\,|||\eta|||\,\|\chi\|_{\EGn}.
\]
By \eqref{eq:cea1}, $\|\chi\|_{\EGn}\le\sqrt{10}\,|||\eta|||$, and the triangle inequality gives
$\|p-P\|_{\EGn}\le\|\eta\|_{\EGn}+\|\chi\|_{\EGn}\le(1+\sqrt{10})|||p-z|||$. Taking the infimum over $z\in\Vnh$ completes the proof.
\end{proof}
Theorem~\ref{thm:cea} requires only the regularity $p|_T\in H^{3/2+\epsilon}(T)$ used for the consistency, and no $H^2$ regularity of $p$; the latter is used only in Theorem~\ref{thm:error} below to obtain a convergence rate. The constant $1+\sqrt{10}$ does not involve $C_{tr}(n,h)$: the only edge term containing a trace of a neural function, the $\xi$-term, is handled by Lemma~\ref{lem:edge}, the GLS term and, on $\Gamma_N$, the Neumann least-squares term, at the price of the terms $\tau^{-1/2}\|\eta\|_0$ and $a_7$ in \eqref{eq:augnorm}, which are controlled by the $L^2$ error of the simultaneous approximant and by the piecewise constant enrichment (Lemma~\ref{lem:elem} and the proof of Theorem~\ref{thm:error}). For mixed boundary conditions, the Neumann least-squares term is needed only to control the last term in \eqref{eq:xiN}, which contains the normal derivative of a neural function on $\Gamma_N$. Without it ($\gamma=0$), the Cauchy-Schwarz inequality with the weights $\hat h_e^{\mp1/2}$ and \eqref{eq:cn} give
\[
|\langle\eta,\kappa\nabla\chi_n\cdot\bn\rangle_{\mathcal{E}_h^{N,\partial}}|\le C_N(n,h)^{1/2}\Big(\sum_{e\in\mathcal{E}_h^{N,\partial}}\hat h_e^{-1}\|\eta\|_{0,e}^2\Big)^{1/2}\|\chi\|_{\EGn},
\]
so that Theorem~\ref{thm:cea} and the estimates below hold with $\sqrt{10}$ replaced by a constant proportional to $(1+C_N(n,h))^{1/2}$, which, like $C_{tr}(n,h)$, may grow with $n$. The term does not affect the conservation (\S\ref{sec3}), since $\nabla\psi_T=0$. 

We shall present the convergence rate for the full elliptic regularity case first. 
\begin{theorem}\label{thm:error}
Let $\Gamma_D\ne\emptyset$, $\gamma>0$ if $\Gamma_N\ne\emptyset$, $k\ge2$, $2\le r\le r_{\max}$, $p\in H^r(\Omega)$, and let $\tau$ satisfy \eqref{cond}. Let $\xi=1$ and $\alpha>0$. Then
\begin{equation}\label{eq:errest}
\|p-P\|_{\EGn}\lesssim\inf_{v\in\Lkn}\big(\delta^{-1}\|p-v\|_{0}+\|p-v\|_{H^1(\Omega)}+\delta|p-v|_{H^2(\Omega)}\big)
\lesssim\delta^{\,r-1}\|p\|_{H^r(\Omega)},
\end{equation}
where the constant depends on $\alpha$, $\kappa$, the constants in \eqref{cond} and the shape regularity of $\Th$, but not on
$n$ and $h$. In particular, the error in the EG norm is of $H^1$ optimal rate
$O(n^{-(r-1)/d})$ for every shape-regular subdivision, coarse or fine.
\end{theorem}

\begin{proof}
We proceed in four steps.
\emph{Step 1:} Fix $v\in\Lkn$ and set
\[
\rho:=p-v\in H^2(\Omega),\qquad z:=v+Q_h^0\rho\in\Vnh,\qquad \eta:=p-z=\rho-Q_h^0\rho ;
\]
for $v=I_np$, we have $z=\Pi_{n,h}p$. Since $p\in H^r(\Omega)\subset H^2(\Omega)$ and $\kappa$ is continuous, $p$ satisfies the assumptions of Theorem~\ref{thm:cea}. Hence, with $\chi:=z-P$,
\begin{equation}\label{eq:chibound}
\|\chi\|_{\EGn}\le\sqrt{10}\,|||\eta|||,\qquad \|p-P\|_{\EGn}\le\big(1+\sqrt{10}\big)|||\eta||| ,
\end{equation}
and it remains to bound the seven parts $a_1,\dots,a_7$ of $|||\eta|||$ defined in the proof of Theorem~\ref{thm:cea}. We use that $\nabla_h\eta=\nabla\rho$ and $\nabla\cdot\kappa\nabla\eta=\nabla\cdot\kappa\nabla\rho$ on each element, since $Q_h^0\rho$ is piecewise constant. \emph{Step 2:} By \eqref{eq:kappa-bd},
\[
a_1\le\kappa_{\rm max}^{1/2}\|\nabla\rho\|_0,\qquad
a_3\le\tau^{1/2}\big(\kappa_{\rm max}|\rho|_{H^2}+\|\nabla\kappa\|_{L^\infty}\|\nabla\rho\|_0\big).
\]
\emph{Step 3:} By \eqref{eq:jumpPi}, whose proof uses only the continuity of $v$, and by \eqref{eq:etaedge} on the Neumann edges,
\[
a_2\lesssim\alpha^{1/2}\big(\|\nabla\rho\|_0+\delta^{-1}\|\rho\|_0\big),\qquad a_7\lesssim\gamma^{-1/2}\big(\|\nabla\rho\|_0+\delta^{-1}\|\rho\|_0\big).
\]
Moreover, $\|\eta\|_{0,T}\le\min\{\|\rho\|_{0,T},C_Ph_T\|\nabla\rho\|_{0,T}\}$ by \eqref{eq:poinc}, and \eqref{cond} gives $a_5=\tau^{-1/2}\|\eta\|_0\lesssim\min\{h,\delta\}^{-1}\min\{\|\rho\|_0,h\|\nabla\rho\|_0\}\le\|\nabla\rho\|_0+\delta^{-1}\|\rho\|_0$.
\emph{Step 4:} Since $\rho\in H^2(\Omega)$, $\av{\nabla\eta}=\nabla\rho$ on each edge, and \eqref{eq:edgegrad} gives
\[
a_4\lesssim(\kappa_{\rm max}/\alpha)^{1/2}\big(\|\nabla\rho\|_0+\delta|\rho|_{H^2}\big),\qquad
a_6\lesssim\gamma^{1/2}\kappa_{\rm max}\big(\|\nabla\rho\|_0+\delta|\rho|_{H^2}\big).
\]
Collecting Steps 2-4 and using \eqref{cond}, we arrive at
\begin{equation}\label{eq:etabound}
|||\eta|||\lesssim\delta^{-1}\|\rho\|_0+\|\rho\|_{H^1}+\delta|\rho|_{H^2},
\end{equation}
where the hidden constant depends on $\alpha$, $\gamma$, $\kappa_{\rm max}$, $\|\nabla\kappa\|_{L^\infty}$, $C_P$ and the constants in \eqref{cond}, but not on $h$. Together with \eqref{eq:chibound} and the infimum over $v$, this gives the first inequality in \eqref{eq:errest}. Choosing $v=I_np$, Proposition~\ref{prop:simul} gives
\[
\delta^{-1}\|p-I_np\|_0+\|p-I_np\|_{H^1}+\delta|p-I_np|_{H^2}\lesssim\delta^{r-1}\|p\|_{H^r},
\]
which is the second inequality in \eqref{eq:errest}. This completes the proof.  
\end{proof}

\begin{remark}[Role of the enrichment]
The piecewise constant part $Q_h^0\rho$ of $z$ enters the error analysis only through the Poincar\'e inequality \eqref{eq:poinc} on elements with $h_T\le\delta$ in Lemma~\ref{lem:elem}. Without it, i.e., for $z=v\in\Lkn$ and $\eta=\rho$, the jumps on interior edges vanish, but on Dirichlet edges $\hat h_e^{-1}\|\rho\|_{0,e}^2\lesssim\hat h_e^{-1}(h_T^{-1}\|\rho\|_{0,T}^2+\|\rho\|_{0,T}\|\nabla\rho\|_{0,T})$ exceeds the optimal order $\delta^{2(r-1)}$ by the factor $\delta^2/h^2$ when $h\ll\delta$. On a coarse subdivision, $h\gtrsim\delta$, the enrichment is not needed for the accuracy, and its role is the local conservation (\S\ref{sec3}).
\end{remark}
The $H^2$ regularity in Theorem~\ref{thm:error} can be relaxed. The norm $|||\cdot|||$ contains $\nabla\cdot\kappa\nabla\eta$ and the edge traces of $\nabla\eta$, but no full $H^2$ norm, and $\nabla\cdot\kappa\nabla p=-f\in L^2(\Omega)$ for every solution. We now consider the reduced regularity case. 
\begin{theorem}\label{thm:lowreg}
Let the assumptions of Theorem~\ref{thm:error} hold, except that $p\in H^s(\Omega)$ with $\frac32<s\le2$. Then
\begin{equation}\label{eq:lowreg}
\|p-P\|_{\EGn}\lesssim\delta^{\,s-1}\big(\|p\|_{H^s(\Omega)}+\|f\|_0\big),
\end{equation}
which is of optimal order $O(n^{-(s-1)/d})$ in the EG norm. 
\end{theorem}
\begin{proof}
We follow the proof of Theorem~\ref{thm:error}, with $p$ replaced in the approximation step by a mollified function. \emph{Step 1:} Since $s>\frac32$, $\nabla p$ has $L^2$ traces on all edges, and $p$ satisfies the assumptions of Theorem~\ref{thm:cea}. Let
\begin{equation*}
p_\delta:=(\omega_\delta*Ep)|_\Omega,\qquad v:=I_np_\delta\in\Lkn,
\end{equation*}
where $E$ is Stein's extension operator \cite{stein1970singular} and $\omega_\delta$ is a standard mollifier with support in the ball of radius $\delta$. As in the proof of Theorem~\ref{thm:error}, set $\rho:=p-v$, $z:=v+Q_h^0\rho\in\Vnh$ and $\eta:=p-z=\rho-Q_h^0\rho$. By \eqref{eq:chibound}, it suffices to bound the parts $a_1,\dots,a_7$ of $|||\eta|||$. \emph{Step 2:} By the standard properties of mollification,
\begin{eqnarray}\label{eq:moll}
&& \|p-p_\delta\|_{0}\lesssim\delta^{s}\|p\|_{H^s}, \qquad \|p-p_\delta\|_{H^1}\lesssim\delta^{s-1}\|p\|_{H^s},\\ 
&& |p-p_\delta|_{H^s}\lesssim\|p\|_{H^s},\qquad \|p_\delta\|_{H^2}\lesssim\delta^{s-2}\|p\|_{H^s}. \nonumber
\end{eqnarray}
By Proposition~\ref{prop:simul} with $r=2$ and \eqref{eq:moll},
\begin{equation}\label{eq:pdv}
\|p_\delta-v\|_{H^m}\lesssim\delta^{2-m}\|p_\delta\|_{H^2}\lesssim\delta^{s-m}\|p\|_{H^s},\quad m=0,1,2,
\end{equation}
and hence $\|v\|_{H^2}\lesssim\delta^{s-2}\|p\|_{H^s}$.
In particular, $\|\nabla\rho\|_0\le\|\nabla(p-p_\delta)\|_0+\|\nabla(p_\delta-v)\|_0\lesssim\delta^{s-1}\|p\|_{H^s}$ and, in the same way, $\|\rho\|_0\lesssim\delta^{s}\|p\|_{H^s}$.
\emph{Step 3:} The bounds for $a_1$, $a_2$, $a_5$ and $a_7$ in the proof of Theorem~\ref{thm:error} use only $\|\nabla\rho\|_0+\delta^{-1}\|\rho\|_0$. Hence, by Step 2,
\[
a_1+a_2+a_5+a_7\lesssim\|\nabla\rho\|_0+\delta^{-1}\|\rho\|_0\lesssim\delta^{s-1}\|p\|_{H^s}.
\]
\emph{Step 4:} Here $|\rho|_{H^2}$ is not available, but $\nabla\cdot\kappa\nabla p=-f$. Hence, on each element, $\nabla\cdot\kappa\nabla\eta=-f-\nabla\cdot\kappa\nabla v$, and by \eqref{eq:pdv}, $\tau^{1/2}\lesssim\delta\le1$ and $s\le2$,
\[
a_3\le\tau^{1/2}\big(\|f\|_0+C_\kappa\|v\|_{H^2}\big)\lesssim\delta\|f\|_0+\delta^{s-1}\|p\|_{H^s}\le\delta^{s-1}\big(\|f\|_0+\|p\|_{H^s}\big).
\]
\emph{Step 5:} As before, $\av{\nabla\eta}=\nabla\rho$ on each edge, and we split $\rho=(p-p_\delta)+(p_\delta-v)$. For the smooth part, i.e., $(p_\delta - v)$, \eqref{eq:edgegrad} and \eqref{eq:pdv} give
\begin{equation*}
\sum_e\hat h_e\|\nabla(p_\delta-v)\|_{0,e}^2\lesssim\|\nabla(p_\delta-v)\|_0^2+\delta\|\nabla(p_\delta-v)\|_0|p_\delta-v|_{H^2}\lesssim\delta^{2(s-1)}\|p\|_{H^s}^2 .
\end{equation*}
For the rough part, i.e., $(p - p_\delta)$, we use a standard dyadic (Littlewood--Paley type) decomposition, see, e.g., \cite{devore1993constructive}. Let $\delta_j:=2^{-j}\delta$ and $g_j:=p_{\delta_{j+1}}-p_{\delta_j}$, $j\ge0$, where $p_\varepsilon:=(\omega_\varepsilon*Ep)|_\Omega$. Then $p-p_\delta=\sum_{j\ge0}g_j$ in $H^s(\Omega)$, hence, since $s>\frac32$, also for the traces of the gradients on every edge, and the standard properties of mollification \cite{devore1993constructive}, give 
\begin{equation} 
\|\nabla g_j\|_0\lesssim\delta_j^{s-1}\|p\|_{H^s},\qquad |g_j|_{H^2}\lesssim\delta_j^{s-2}\|p\|_{H^s}.
\end{equation} 
By \eqref{eq:fmtrace} with $\phi=\partial_ig_j$ and $\hat h_e\lesssim\min\{h_T,\delta\}$,
\begin{eqnarray*}
\sum_e\hat h_e\|\nabla g_j\|_{0,e}^2 &\lesssim& \|\nabla g_j\|_0^2+\delta\|\nabla g_j\|_0|g_j|_{H^2} \\
&\lesssim& \big(\delta_j^{2s-2}+\delta\,\delta_j^{2s-3}\big)\|p\|_{H^s}^2\lesssim2^{-j(2s-3)}\delta^{2s-2}\|p\|_{H^s}^2 .
\end{eqnarray*}
Since $s>\frac32$, the triangle inequality for the weighted edge norm and the geometric series give
\[
\Big(\sum_e\hat h_e\|\nabla(p-p_\delta)\|_{0,e}^2\Big)^{1/2}\le\sum_{j\ge0}\Big(\sum_e\hat h_e\|\nabla g_j\|_{0,e}^2\Big)^{1/2}\lesssim\frac{\delta^{s-1}}{1-2^{-(s-3/2)}}\|p\|_{H^s}.
\]
Hence $a_4\lesssim\delta^{s-1}\|p\|_{H^s}$ and, in the same way, $a_6\lesssim\delta^{s-1}\|p\|_{H^s}$.
Collecting Steps 3--5 gives $|||\eta|||\lesssim\delta^{s-1}(\|p\|_{H^s}+\|f\|_0)$, and \eqref{eq:chibound} gives \eqref{eq:lowreg}, which completes the proof. 
\end{proof}

\begin{remark}[Comparison with \cite{liu2026neural}]\label{rem:liu}
For the GLS-unsymmetric Nitsche method on $\Lkn$ with fixed $\tau$, \cite{liu2026neural} proves a C\'ea estimate with constant $O(\tau^{-1/2})$ and the rate $n^{-(r-2)/d}$ in the $H^2$-type energy norm for $p\in H^r$, $r\ge2$. On the other hand, the current work established that the EG norm error is of the optimal $H^1$ rate, although the EG norm contains the $H^2$-type GLS term and it remains true for solutions in $H^s$, $s>\frac32$, with the rate $n^{-(s-1)/d}$ (Theorem~\ref{thm:lowreg}). In addition, the ENN flux is locally conservative. For fixed $\tau$, the rate in \cite{liu2026neural} is optimal for its $H^2$-type energy norm. With $\tau$ scaled as in \eqref{cond} and a simultaneous approximant, the same mechanism yields the optimal $H^1$ rate, and the argument above also applies to the method of \cite{liu2026neural}. If $\tau\simeq\delta^2$, then Proposition~\ref{prop:simul} and the multiplicative trace inequality $\|\phi\|_{0,\partial\Omega}^2\lesssim\|\phi\|_0\|\phi\|_{H^1}$ for all $\phi \in H^1(\Omega)$, give the optimal $H^1$ rate: $\|p-P_L\|_L\lesssim\delta^{r-1}\|p\|_{H^r(\Omega)}$, where $\|\cdot\|_L$ is the energy norm of \cite{liu2026neural}.
\end{remark}

\paragraph{The method of \cite{liu2026neural}} For $\Gamma_D=\partial\Omega$ and the diffusion operator considered here (\cite{liu2026neural} treats more general second-order operators), the GLS--unsymmetric Nitsche method of \cite{liu2026neural} seeks $P_L\in\Lkn$ with $B_L(P_L,w)=\ell_L(w)$ for all $w\in\Lkn$, where
\[
\begin{aligned}
B_L(v,w):={}&(\kappa\nabla v,\nabla w)-\langle\kappa\nabla v\cdot\bn,w\rangle_{\partial\Omega}+\langle v,\kappa\nabla w\cdot\bn\rangle_{\partial\Omega}\\
&+\beta_L\langle v,w\rangle_{\partial\Omega}+\tau(\nabla\cdot\kappa\nabla v,\nabla\cdot\kappa\nabla w),
\end{aligned}
\]
$\ell_L$ is the consistent load functional, and $\beta_L,\tau>0$ are fixed. This is \eqref{eq:weak} with $\xi=1$ restricted to $\Lkn$, where the interior edge terms vanish, with the penalty $\alpha\hat h_e^{-1}$ replaced by $\beta_L$. Let $\|v\|_L^2:=\|\kappa^{1/2}\nabla v\|_0^2+\beta_L\|v\|_{0,\partial\Omega}^2+\tau\|\nabla\cdot\kappa\nabla v\|_0^2$.

\begin{corollary}\label{cor:liu}
Let $\Gamma_D=\partial\Omega$, $k\ge2$, $2\le r\le r_{\max}$ and $p\in H^r(\Omega)$. Then
\begin{equation}\label{eq:liu}
\|p-P_L\|_L\lesssim\delta^{r-1}\Big(1+(\beta_L\delta)^{1/2}+(\beta_L\delta)^{-1/2}+\frac{\tau^{1/2}}{\delta}+\frac{\delta}{\tau^{1/2}}\Big)\|p\|_{H^r(\Omega)} .
\end{equation}
In particular, the error is of the optimal order $O(n^{-(r-1)/d})$ if $\beta_L\simeq\delta^{-1}$ and $\tau\simeq\delta^2$, and only of the order $O(n^{-(r-2)/d})$ if $\beta_L$ and $\tau$ are fixed.
\end{corollary}
\begin{proof}
The proof of Theorem~\ref{thm:cea} uses only the coercivity, the Galerkin orthogonality and Lemma~\ref{lem:edge} with $\sigma=\kappa\nabla\chi_n$. It therefore applies on the subspace $\Lkn$, where $\chi=\chi_n$ and $\jump\chi=0$ on interior edges, with $\alpha\hat h_e^{-1}$ replaced by $\beta_L$; here $B_L(\chi,\chi)=\|\chi\|_L^2$. With $\eta:=p-I_np$, it gives $\|p-P_L\|_L\lesssim M(\eta)$, where
\[
M(\eta):=\|\kappa^{1/2}\nabla\eta\|_0+\beta_L^{1/2}\|\eta\|_{0,\partial\Omega}+\beta_L^{-1/2}\|\kappa^{1/2}\nabla\eta\|_{0,\partial\Omega}+\tau^{1/2}\|\nabla\cdot\kappa\nabla\eta\|_0+\tau^{-1/2}\|\eta\|_0
\]
collects the terms $a_1,\dots,a_5$ of $|||\eta|||$. By Proposition~\ref{prop:simul}, $\|\eta\|_{H^m}\lesssim\delta^{r-m}\|p\|_{H^r}$ for $m=0,1,2$, and by the trace inequality $\|\phi\|_{0,\partial\Omega}^2\lesssim\|\phi\|_0\|\phi\|_{H^1}$ \cite[Theorem~1.6.6]{brenner2008mathematical}, $\|\eta\|_{0,\partial\Omega}\lesssim\delta^{r-1/2}\|p\|_{H^r}$ and $\|\nabla\eta\|_{0,\partial\Omega}\lesssim\delta^{r-3/2}\|p\|_{H^r}$. Inserting these bounds into $M(\eta)$ gives \eqref{eq:liu}.
\end{proof}
The multiplicative trace bounds are where the simultaneous approximation is essential: $\|\nabla\eta\|_{0,\partial\Omega}$ requires $\nabla\eta$ to be small in $L^2$ and in $H^1$ for the same $\eta$. The scalings $\beta_L\simeq\delta^{-1}$ and $\tau\simeq\delta^2$ are those of the ENN method with $\hat h_e=\delta$ on a coarse subdivision; the ENN method adds the local conservation on an arbitrary subdivision. 



Finally, we consider the pure Neumann condition, $\Gamma_D = \emptyset$. Solvability requires 
$\int_\Omega f\,dx=\int_{\partial\Omega}g_N\,ds$, and $p$ is determined up to a constant. Since $1\in\Vnh$,
$A_\xi(\cdot,1)=0$ and $F_\xi(1)=0$, and we solve \eqref{eq:weak} in $\{v\in\Vnh:\int_\Omega v\,dx=0\}$, where $\|\cdot\|_{\EGn}$ is a norm by the broken Poincar\'e inequality \cite{hong2025sharpness}.
\begin{theorem}[pure Neumann problem]\label{thm:neumann}
Let $\Gamma_D=\emptyset$, $\xi=1$, $\alpha>0$, $\tau\ge0$ and $\gamma=0$. Then
\[
\|p-P\|_{\EGn}\lesssim\|\nabla(p-I_np)\|_0+(\delta+\tau^{1/2})|p-I_np|_{H^2}\lesssim\big(\delta^{r-1}+\tau^{1/2}\delta^{r-2}\big)\|p\|_{H^r}
\]
for every shape-regular subdivision. In particular, the rate is of optimal $H^1$ order for $0\le\tau\lesssim\delta^2$.
\end{theorem}
\begin{proof}
Use the continuous approximant $I_np$ (shifted by a constant) instead of $\Pi_{n,h}p$. Then
$\jump{p-I_np}=0$ on every edge and there are no Dirichlet edges, so the $\xi$- and penalty terms of
$A_1(p-I_np,\chi)$ vanish. In the proof of Theorem~\ref{thm:cea}, only the volume term, the term $\langle\av{\kappa\nabla\eta}_{\beta_e},\jump\chi\rangle$ and the GLS term remain, so that the terms $a_5$ and $a_7$ are not needed. The first two are bounded by \eqref{eq:edgegrad} as in Step~4 of the proof of Theorem~\ref{thm:error}, and the GLS term satisfies $\tau(\nabla\cdot\kappa\nabla\eta,\nabla\cdot\kappa\nabla\chi)\le C_\kappa\tau^{1/2}|\eta|_{H^2}\|\chi\|_{\EGn}$. This completes the proof.
\end{proof}

Thus the GLS term is needed only because of the Dirichlet data. Namely, the interior edges never require the GLS term. 
\section{Locally Conservative Flux and its Error Analysis}\label{sec3}
The ENN numerical flux is defined as follows:
\begin{subequations}\label{eq:flux}
\begin{align}
\bU|_T&=-\kappa\nabla P&&\text{on each element},\\
\bU\cdot \bn&=-\kappa_e\av{\nabla P}\cdot \bn + \alpha \hat h_e^{-1}\kappa_e\jump P\cdot \bn&&\text{on each interior edge }e,\\
\bU\cdot \bn&=g_N&&\text{on a Neumann boundary edge},\\
\bU\cdot \bn&=-\kappa\nabla P\cdot \bn + \alpha \hat h_e^{-1}\kappa_e(P-g_D)&&\text{on a Dirichlet boundary edge},
\end{align}
\end{subequations}
where $\bn$ is a fixed unit normal of $e$ and $\jump P\cdot \bn$ is the corresponding difference of the
traces of $P$. As discussed in \cite{ern2007accurate,lee2016locally}, it is then easy to establish the following local
conservation.

\begin{lemma}\label{lem:conservation}
For every $\xi\in\{-1,0,1\}$ and all $\alpha,\tau,\gamma\ge0$, the flux $\bU$ is conservative in the sense
that
\begin{equation}\label{eq:conservation}
\int_{\partial T}\bU\cdot \bn_T\,ds=\int_Tf\,dx,\qquad\forall T\in\Th .
\end{equation}
\end{lemma}
\begin{proof}
Take $w=\psi_T$ in \eqref{eq:weak}. Since $\nabla\psi_T=0$ and $\nabla\cdot\kappa\nabla\psi_T=0$, only the
$\av{\cdot}$-term and the penalty term of $A_\xi(P,\psi_T)$ and the terms $(f,\psi_T)$, $\langle g_N,\psi_T\rangle$ and the
$\alpha$-term of $F_\xi(\psi_T)$ remain. Since $\jump{\psi_T}=\bn_T$ on $\partial T\cap\mathcal{E}_h^1$ and
$\psi_T=1$ on $\partial T\cap\Gamma_N$, these terms give, by \eqref{eq:kappae} and the definition \eqref{eq:flux},
exactly $\sum_{e\subset\partial T}\int_e\bU\cdot\bn_T\,ds=\int_Tf\,dx$; the integrand on an interior edge does not
depend on the orientation of $\bn$, since $\av{\nabla P}\cdot\bn$ and $\jump P\cdot\bn$ both change sign with $\bn$.
\end{proof}

\begin{theorem}\label{thm:fluxerror}
Let the assumptions of Theorem~\ref{thm:error} hold, and let $\bu=-\kappa\nabla p$. Then
\begin{subequations}\label{eq:fluxest}
\begin{align}
\|\bu-\bU\|_{0,\Th}&\lesssim\|p-P\|_{\EGn}\lesssim\delta^{r-1}\|p\|_{H^r},\label{eq:fluxest-a}\\
\sum_{e\in\mathcal{E}_h}\hat h_e\|\bu\cdot \bn-\bU\cdot \bn\|_{0,e}^2&\lesssim\big(1+\alpha+C_{tr}(n,h)\big)\,\delta^{2(r-1)}\|p\|_{H^r}^2 .\label{eq:fluxest-b}
\end{align}
\end{subequations}
\end{theorem}
\begin{proof}
On each $T$, $\bu-\bU=-\kappa\nabla(p-P)$, which gives \eqref{eq:fluxest-a} by Theorem~\ref{thm:error}. On a
Neumann edge, $\bu\cdot\bn=\bU\cdot\bn$. On the edges of $\mathcal{E}_h^1$, since $\jump p=0$ on interior edges,
$p=g_D$ on $\Gamma_D$ and $\av{\nabla P}=\nabla P_n$, \eqref{eq:flux} gives
$\bu\cdot \bn-\bU\cdot \bn=-\kappa\nabla(p-P_n)\cdot \bn-\alpha\hat h_e^{-1}\kappa\jump{p-P}\cdot \bn$.
With the notation of the proof of Theorem~\ref{thm:error} for $v=I_np$, $P_n=v-\chi_n$, so $p-P_n=\rho+\chi_n$.
By \eqref{eq:edgegrad} and Lemma~\ref{lem:ctr},
\[
\sum_{e\in\mathcal{E}_h^1}\hat h_e\|\kappa\nabla(\rho+\chi_n)\|_{0,e}^2\lesssim\|\nabla\rho\|_0^2+\delta^2|\rho|_{H^2}^2+C_{tr}(n,h)\|\chi\|_{\EGn}^2,
\]
and $\alpha^2\sum_{e\in\mathcal{E}_h^1}\hat h_e^{-1}\|\kappa\jump{p-P}\|_{0,e}^2\le\alpha\kappa_{\rm max}\|p-P\|_{\EGn}^2$. Proposition~\ref{prop:simul},
\eqref{eq:chibound}, \eqref{eq:etabound} and Theorem~\ref{thm:error} bound all terms by $\delta^{2(r-1)}\|p\|_{H^r}^2$.
\end{proof}
The bound \eqref{eq:fluxest-b} is of optimal order up to $C_{tr}(n,h)$, which is not known to be bounded uniformly in $n$; in our computations it stays between $1.4$ and $2.7$.

\section{Numerical Experiments}\label{numerical_experiment}
In this section, we consider the model problem \eqref{eq:strong} on $\Omega = (0,1)^2$ with $\Gamma_D=\{0,1\}\times[0,1]$, $\Gamma_N=[0,1]\times\{0,1\}$, $g_D=1$ on $x=0$, $g_D=0$ on $x=1$, and $g_N=0$. The trial space is $V_{n,h}^{\enn}=L^k_n+\mathcal{M}^0(\mathcal{T}_h)$ with the neurons \eqref{eq:phij}, and the unknown coefficients of $P$ are obtained from the Galerkin system \eqref{eq:weak}; no collocation or iterative training is involved. Because of the redundancy discussed in \S\ref{model}, the linear system is singular, and it is solved by an SVD-based least-squares solver.

We use the manufactured solution $p(x,y)=1-x+0.2\sin(\pi x)\cos(2\pi y)$, which satisfies the above boundary conditions, the smooth heterogeneous conductivity: 
\begin{equation*}
\kappa(x,y)=\exp\left(
0.45\sin(2\pi(x+y))-0.30\sin(2\pi(2x+y))
+0.20\sin(2\pi(x+3y))
\right),
\end{equation*}
and the corresponding source $f=-\nabla\cdot(\kappa\nabla p)$. The Neumann least-squares term is not used ($\gamma=0$); this case is covered by the remark after Theorem~\ref{thm:cea}, with a constant depending on $C_N(n,h)$. Since $p$ is smooth, the rates are limited by $r=r_{\max}=\frac{d+2k+1}{2}$, i.e., $r=4.5$ for ReLU$^3$ and $r=3.5$ for ReLU$^2$. In terms of $n$, Theorem~\ref{thm:error} predicts the rates $n^{-(r-1)/2}$, i.e., $1.75$ (ReLU$^3$) and $1.25$ (ReLU$^2$), for $\tau\simeq n^{-1}$ on any grid.

\subsection{Numerical Setting}
We use ReLU$^k$ neurons with $k\in\{2,3\}$. The parameters $\theta_\ell=(\omega_\ell,b_\ell)$ are placed on $\mathbb S^2$ by the
Fibonacci sphere rule 
\begin{equation*}
b_\ell=1-\frac{2\ell+1}{n}, \omega_\ell=\sqrt{1-b_\ell^2}\,(\cos(\ell\pi(3-\sqrt5)),\sin(\ell\pi(3-\sqrt5))), \ell=0,\ldots,n-1,
\end{equation*} 
and remain fixed throughout the computation.  Unless stated otherwise, the centre is $x_c=(\frac12,\frac12)$; see \S\ref{sec:centre}. The enrichment mesh is a uniform square grid.

We use the NIPG form $\xi=1$, the penalty parameter $\alpha=10$, i.e., the edge penalty $10\,\hat h_e^{-1}$, and the
Galerkin least-squares parameter $\tau=\delta^2=n^{-1}$.  Since the neurons have kinks (jumps of the $k$-th derivative) along hyperplanes that are not aligned with $\mathcal T_h$, all element and edge
integrals in $A_\xi$ and $F_\xi$ are evaluated by a composite three-point Gauss rule
in each coordinate direction on a fixed uniform background grid of about $96\times96$ cells, independent of the enrichment grid; the same rule is used to evaluate the errors.  The relative SVD cutoff is $10^{-12}$.  We report the $L^2$ error $\|p-P\|_0$ and the EG error \eqref{eq:EG} (computed with the unweighted gradient term).
The mass conservation loss is $\mathcal E_{\rm mass}=\max_{T\in\mathcal T_h}|r_T|$ with $r_T=\int_{\partial T}\bU\cdot\bn_T\,ds-\int_T f\,dx$, where the numerical flux \eqref{eq:flux} is integrated by the same edge quadrature as in the assembly.
The last row of each table gives the rate obtained by a least-squares fit of $\log E_n$ against $\log n$.

\subsection{Placement of the inner parameters and the $H^1$-optimal regime}\label{sec:centre}
With $x_c=0$, the hyperplanes of many neurons do not cross $\Omega=(0,1)^2$. Moving the centre to $x_c=(\frac12,\frac12)$ increases the number $n_{\rm eff}$ of active neurons and makes the convergence more regular. For example, for ReLU$^3$ on the $12\times12$ grid with $n=64,144,256,576,1024$, the observed EG orders are $1.75, 1.32, 1.60, 1.30$ (fitted rate $1.52$) for $x_c=0$ and $1.94, 1.55, 1.70, 1.49$ (fitted rate $1.69$) for $x_c=(\frac12,\frac12)$, and the EG errors are smaller by a factor of up to $2.4$. We use $x_c=(\frac12,\frac12)$ below.

Theorem~\ref{thm:error} gives the $H^1$-optimal rate on any subdivision for $\min\{h,\delta\}^2\lesssim\tau\lesssim\delta^2$; we use $\tau=\delta^2$, which is admissible for every $h$. In Table~\ref{tab:h1} we refine the enrichment grid together with the network, $h=n^{-1/2}=\delta$, so that $\hat h_e=h_e$. The fitted EG rates are $1.68$ for ReLU$^3$ (theory $1.75$) and $1.22$ for ReLU$^2$ (theory $1.25$). The $L^2$ rate for ReLU$^3$ is $2.26$, the optimal approximation rate, although we prove no $L^2$ estimate.

\begin{table}[htbp]
\centering
\caption{Errors with the grid refined together with the network; NIPG, $h=n^{-1/2}=\delta$, $\tau=\delta^2$, $x_c=(\frac12,\frac12)$. Predicted rate of the EG error: $1.75$ (ReLU$^3$) and $1.25$ (ReLU$^2$).}
\label{tab:h1}
\begingroup
\scriptsize
\setlength{\tabcolsep}{3pt}
\begin{tabular}{cc cc cc}
\hline
 & & \multicolumn{2}{c}{ReLU$^3$} & \multicolumn{2}{c}{ReLU$^2$}\\
Neurons & Grid & $\|p-P\|_{0}$ & $\|p-P\|_{\rm EG}$ & $\|p-P\|_{0}$ & $\|p-P\|_{\rm EG}$\\
\hline
64 & $8\times8$ & $4.57\times10^{-3}$ & $1.40\times10^{-1}$ & $1.22\times10^{-2}$ & $3.45\times10^{-1}$ \\
144 & $12\times12$ & $4.98\times10^{-4}$ & $2.95\times10^{-2}$ & $2.23\times10^{-3}$ & $1.35\times10^{-1}$ \\
256 & $16\times16$ & $1.46\times10^{-4}$ & $1.21\times10^{-2}$ & $8.81\times10^{-4}$ & $6.70\times10^{-2}$ \\
576 & $24\times24$ & $2.44\times10^{-5}$ & $3.04\times10^{-3}$ & $1.79\times10^{-4}$ & $2.45\times10^{-2}$ \\
1024 & $32\times32$ & $8.53\times10^{-6}$ & $1.29\times10^{-3}$ & $6.04\times10^{-5}$ & $1.19\times10^{-2}$ \\
\hline
fit &  & 2.26 & 1.68 & 1.90 & 1.22 \\
\hline
\end{tabular}
\endgroup
\end{table}

Table~\ref{tab:coarse} keeps the enrichment grid fixed at $4\times4$, so that $h=\frac14\gg\delta$ for all $n$. The errors are practically the same as in Table~\ref{tab:h1}, e.g., $1.29\times10^{-3}$ for ReLU$^3$ and $n=1024$ with $16$ instead of $1024$ elements, and the fitted EG rates, $1.69$ and $1.22$, agree with the theory; the results on the $8\times8$ grid are almost identical. The mass conservation loss is at round-off level, $\mathcal E_{\rm mass}\le2\times10^{-10}$. Thus a coarse subdivision provides the local conservation, while the network provides the accuracy.

\begin{table}[htbp]
\centering
\caption{Errors on a fixed coarse grid; NIPG, $4\times4$ grid, $\tau=\delta^2$, $x_c=(\frac12,\frac12)$. Predicted rate of the EG error: $1.75$ (ReLU$^3$) and $1.25$ (ReLU$^2$).}
\label{tab:coarse}
\begingroup
\scriptsize
\setlength{\tabcolsep}{3pt}
\begin{tabular}{c cc cc}
\hline
 & \multicolumn{2}{c}{ReLU$^3$} & \multicolumn{2}{c}{ReLU$^2$}\\
Neurons & $\|p-P\|_{0}$ & $\|p-P\|_{\rm EG}$ & $\|p-P\|_{0}$ & $\|p-P\|_{\rm EG}$\\
\hline
64 & $5.02\times10^{-3}$ & $1.42\times10^{-1}$ & $1.35\times10^{-2}$ & $3.48\times10^{-1}$ \\
144 & $5.68\times10^{-4}$ & $2.98\times10^{-2}$ & $2.52\times10^{-3}$ & $1.36\times10^{-1}$ \\
256 & $1.63\times10^{-4}$ & $1.21\times10^{-2}$ & $9.90\times10^{-4}$ & $6.71\times10^{-2}$ \\
576 & $2.70\times10^{-5}$ & $3.04\times10^{-3}$ & $2.03\times10^{-4}$ & $2.45\times10^{-2}$ \\
1024 & $9.62\times10^{-6}$ & $1.29\times10^{-3}$ & $6.76\times10^{-5}$ & $1.19\times10^{-2}$ \\
\hline
fit & 2.25 & 1.69 & 1.89 & 1.22 \\
\hline
\end{tabular}
\endgroup
\end{table}

\subsection{Comparison with a non-conservative neural method}\label{sec:compare}
We compare the ENN method with the GLS--unsymmetric Nitsche method of \cite{liu2026neural} on $L^k_n$, which is \eqref{eq:weak} restricted to $L^k_n$ with the Nitsche penalty $\alpha\hat h_e^{-1}$ on the $12\times12$ grid (ReLU$^3$, NIPG, $x_c=(\frac12,\frac12)$, the same $\kappa$, boundary conditions, $\alpha$ and $\tau=\delta^2$). For the method of \cite{liu2026neural}, $\bU=-\kappa\nabla P$ and $r_T$ is its mass residual. Table~\ref{tab:three} shows that the two methods have the same accuracy, but only the ENN method is conservative: the mass loss of \cite{liu2026neural} decreases with $n$ but stays far above round-off. With these parameters, both methods converge at the $H^1$-optimal rate (fitted rates $1.69$ and $1.70$), in agreement with Theorem~\ref{thm:error} and Corollary~\ref{cor:liu}. Note that \cite{liu2026neural} uses fixed parameters $\beta_L$ and $\tau$, with which its method converges only at the $H^2$-type rate (Corollary~\ref{cor:liu}); here it is run with the scalings of Corollary~\ref{cor:liu}, so that the comparison isolates the effect of the local conservation.

\begin{table}[htbp]
\centering
\caption{ENN (NIPG) and the method of \cite{liu2026neural}; ReLU$^3$, $x_c=(\frac12,\frac12)$, $12\times12$ grid, $\tau=\delta^2$; $\mathcal E_{\rm mass}$ is computed with the flux of each method.}
\label{tab:three}
\begingroup
\scriptsize
\setlength{\tabcolsep}{3pt}
\begin{tabular}{c cc cc}
\hline
 & \multicolumn{2}{c}{ENN (NIPG)} & \multicolumn{2}{c}{\cite{liu2026neural}}\\
Neurons & $\|p-P\|_{\rm EG}$ & $\mathcal E_{\rm mass}$ & $\|p-P\|_{\rm EG}$ & $\mathcal E_{\rm mass}$\\
\hline
64 & $1.43\times10^{-1}$ & $2\times10^{-12}$ & $1.47\times10^{-1}$ & $1.6\times10^{-2}$ \\
144 & $2.95\times10^{-2}$ & $2\times10^{-11}$ & $2.99\times10^{-2}$ & $3.8\times10^{-3}$ \\
256 & $1.21\times10^{-2}$ & $6\times10^{-12}$ & $1.21\times10^{-2}$ & $2.1\times10^{-3}$ \\
576 & $3.04\times10^{-3}$ & $1\times10^{-11}$ & $3.04\times10^{-3}$ & $4.5\times10^{-4}$ \\
1024 & $1.29\times10^{-3}$ & $8\times10^{-12}$ & $1.29\times10^{-3}$ & $1.8\times10^{-4}$ \\
\hline
\end{tabular}
\endgroup
\end{table}

\section{Conclusion}\label{conclusion}

We developed an enriched linearized neural network method for the Darcy problem, in which a fixed-feature ReLU$^k$ network ($k\ge2$) carries the approximation and piecewise constants on an auxiliary subdivision provide local conservation. Since neural functions have no inverse inequality, we used a Galerkin least-squares term together with an edge identity. This yields a C\'ea-type estimate with a constant independent of $n$, $h$, $\tau$ and $\alpha$, and an error estimate in the EG norm of the optimal $H^1$ order for every subdivision when $\min\{h,n^{-1/d}\}^2\lesssim\tau\lesssim n^{-2/d}$ and the edge terms are scaled with $\min\{h_e,n^{-1/d}\}$, also for solutions in $H^s$ with $s>\frac32$. The numerical flux is locally conservative on every element of any subdivision. The numerical results confirm the predicted rates, also on a fixed $4\times4$ subdivision, and show mass conservation at round-off, in contrast to a non-conservative neural method. Future works include a proof of the boundedness of the neural trace constant, the IIPG and SIPG variants, an $L^2$ error estimate, coefficients that are discontinuous inside elements,  coupled flow and transport problems, and other more complicated PDEs. 

\appendix
\section{Proof of Proposition~\ref{prop:simul}}\label{app:simul}
We use the notation of \cite{liu2025integral}, adapted to avoid clashes with the main text. Let $\{Y_{m,\ell}\}_{\ell=1}^{N(m)}$ be an orthonormal basis of
the
spherical harmonics of degree $m$ with respect to the normalized surface measure. For a function $\Phi$ on $\mathbb{S}^d$, let
$\hat \Phi(m,\ell)=\avint_{\mathbb{S}^d}\Phi\,Y_{m,\ell}$, $\mathcal{Y}_m\Phi=\sum_\ell\hat \Phi(m,\ell)Y_{m,\ell}$ and
$\|\Phi\|_{H^s(\mathbb{S}^d)}^2=\sum_m(m^{2s}+1)\sum_\ell\hat \Phi(m,\ell)^2$. Let $h_\Theta$ and $\underline h_\Theta$ be
the mesh norm and the separation distance of $\{\theta_j\}$; by quasi-uniformity
$h_\Theta\simeq\underline h_\Theta\simeq\delta$. By the Funk--Hecke formula,
\[
\sigma_k(\theta\cdot\eta)=\sum_{m\in \mathcal{I}}\hat\sigma_k(m)\sum_\ell Y_{m,\ell}(\theta)Y_{m,\ell}(\eta),
\]
where
$\mathcal{I}=\{m:\hat\sigma_k(m)\ne0\}$ contains every $m\equiv k+1\pmod2$ and
$\hat\sigma_k(m)^2\simeq(m+1)^{-(d+2k+1)}$ on $\mathcal{I}$ \cite[section~3]{liu2025integral}. We use two results of
\cite{liu2025integral}:
\begin{enumerate}
\item[(Q1)] \cite[Lemma~3.2]{liu2025integral}: there are weights $\varpi_j\ge0$ with $\varpi_j\lesssim h_\Theta^d$ and
  $J=\lfloor C_1h_\Theta^{-1}\rfloor$ such that $\avint_{\mathbb{S}^d}y\,z=\sum_j\varpi_j\,y(\theta_j)\,z(\theta_j)$ for all
  spherical polynomials $y,z$ of degree at most $J$;
\item[(Q2)] \cite[(4.16)--(4.18)]{liu2025integral}: with
  $Z(m)=\big(\sum_\ell Y_{m,\ell}(\theta_i)Y_{m,\ell}(\theta_j)\big)_{i,j}$ and $0\le s\le k$, for every $\mathbf c\in\Reals{n}$,
  \[
  \sum_{m\in \mathcal{I},\,m>J}\hat\sigma_k(m)^2m^{2s}\,\mathbf c^\top Z(m)\mathbf c\lesssim h_\Theta^{2k+1-2s}\|\mathbf c\|_2^2 .
  \]
\end{enumerate}

\begin{lemma}[Simultaneous Approximant on Sphere]\label{lem:sphere}
Let $0<r\le r_{\max}$ and $\Phi\in H^r(\mathbb{S}^d)$ with $\Phi(-\eta)=(-1)^{k+1}\Phi(\eta)$. Define
$c_j:=\varpi_j\sum_{m\in \mathcal{I},\,m\le J}\hat\sigma_k(m)^{-1}(\mathcal{Y}_m\Phi)(\theta_j)$, $\mathbf c:=(c_1,\ldots,c_n)^\top$ and
$\Phi_n:=\sum_jc_j\sigma_k(\theta_j\cdot\eta)$. Then $\|\mathbf c\|_2\lesssim h_\Theta^{(2r-2k-1)/2}\|\Phi\|_{H^r(\mathbb{S}^d)}$ and
\[
\|\Phi-\Phi_n\|_{H^s(\mathbb{S}^d)}\lesssim h_\Theta^{r-s}\|\Phi\|_{H^r(\mathbb{S}^d)}\qquad\text{for every }s\in[0,\min\{k,r\}].
\]
\end{lemma}
\begin{proof}
By parity, $\hat \Phi(m,\cdot)=0$ unless $m\equiv k+1\pmod2$, so $\hat \Phi$ vanishes for $m\notin \mathcal{I}$. Let
$q=\sum_{m\in \mathcal{I},m\le J}\hat\sigma_k(m)^{-1}\mathcal{Y}_m\Phi$, a spherical polynomial of degree at most $J$, so that $c_j=\varpi_jq(\theta_j)$. For
$m\le J$, (Q1) and orthogonality give
\[
\hat \Phi_n(m,\ell)=\hat\sigma_k(m)\sum_j\varpi_jq(\theta_j)Y_{m,\ell}(\theta_j)=\hat\sigma_k(m)\avint_{\mathbb{S}^d}qY_{m,\ell}=\hat \Phi(m,\ell).
\]
Hence $\|\Phi-\Phi_n\|_{H^s}^2\le4I_1+4I_2$ for every $s\ge0$, where, for $s\le r$,
\[
I_1=\sum_{m>J}m^{2s}\sum_\ell\hat \Phi(m,\ell)^2\le J^{2s-2r}\|\Phi\|_{H^r}^2\lesssim h_\Theta^{2r-2s}\|\Phi\|_{H^r}^2,
\]
and, by (Q2) for $s\le k$,
\[
I_2=\sum_{m>J}m^{2s}\sum_\ell\hat \Phi_n(m,\ell)^2=\sum_{m\in \mathcal{I},m>J}\hat\sigma_k(m)^2m^{2s}\,\mathbf c^\top Z(m)\mathbf c\lesssim h_\Theta^{2k+1-2s}\|\mathbf c\|_2^2 .
\]
By (Q1) with $y=z=q$ and since $r\le r_{\max}$,
\begin{align*}
\|\mathbf c\|_2^2&\lesssim h_\Theta^d\sum_j\varpi_jq(\theta_j)^2=h_\Theta^d\avint_{\mathbb{S}^d}q^2
=h_\Theta^d\sum_{m\in \mathcal{I},m\le J}\hat\sigma_k(m)^{-2}\sum_\ell\hat \Phi(m,\ell)^2\\
&\lesssim h_\Theta^dJ^{d+2k+1-2r}\|\Phi\|_{H^r}^2\simeq h_\Theta^{2r-2k-1}\|\Phi\|_{H^r}^2 .
\end{align*}
Combining the bounds gives $\|\Phi-\Phi_n\|_{H^s}^2\lesssim h_\Theta^{2(r-s)}\|\Phi\|_{H^r}^2$. The coefficient
vector $\mathbf c$ does not depend on $s$.
\end{proof}

\begin{proof}[Proof of Proposition~\ref{prop:simul}]
Recall $x_{\rm aug}=(x-x_c,1)\in\Reals{d+1}$ for $x\in\Reals{d}$. Let $B_R:=\{x\in\Reals{d}:|x-x_c|<R\}$, choose
$R_1<R_2$ such that $\overline\Omega\subset B_{R_1}$, and let $\lambda(x)=x_{\rm aug}/|x_{\rm aug}|$,
which maps $B_R$ onto the cap $G_R=\{\eta\in\mathbb{S}^d:\eta_{d+1}>(1+R^2)^{-1/2}\}$. With
$(S_k\Phi)(x)=|x_{\rm aug}|^k\Phi(\lambda(x))$ and $(\tilde S_ku)(\eta)=\eta_{d+1}^ku(\lambda^{-1}(\eta))$ we have
$S_k\tilde S_k=\mathrm{id}$, $S_k[\sigma_k(\theta\cdot\circ)](x)=\sigma_k(\theta\cdot x_{\rm aug})$ by homogeneity, and
$\|S_k\Phi\|_{H^t(B_R)}\simeq\|\Phi\|_{H^t(G_R)}$ for all $t\ge0$ and $R\le R_2$ ($\lambda$ is a smooth
diffeomorphism with bounded derivatives; interpolation for non-integer $t$; cf.\
\cite[Lemma~5.1]{liu2025integral}). Let $v_E$ be Stein's universal extension of $v$ \cite{stein1970singular}, which is
bounded in $H^t$ for all $t\ge0$ simultaneously. Put $\Psi=\tilde S_kv_E$ on $G_{R_2}$, take $\zeta\in
C^\infty(\mathbb{S}^d)$ with $\zeta=1$ on $\overline{G_{R_1}}$ and support in $G_{R_2}$, and let $\Phi=\zeta\Psi$ on the
upper hemisphere and $\Phi(\eta)=(-1)^{k+1}\zeta(-\eta)\Psi(-\eta)$ on the lower one. Then $\Phi\in H^r(\mathbb{S}^d)$ has
the required parity and $\|\Phi\|_{H^r(\mathbb{S}^d)}\lesssim\|v\|_{H^r(\Omega)}$. Apply Lemma~\ref{lem:sphere} and set
$v_n:=S_k\Phi_n=\sum_jc_j\sigma_k(\theta_j\cdot x_{\rm aug})\in\Lkn$. On $\Omega$, $S_k\Phi=S_k\Psi=v_E=v$, so
$v-v_n=S_k(\Phi-\Phi_n)$ and, for $s\in\{0,1,2\}$ (note $s\le\min\{k,r\}$ since $k\ge2$ and $r\ge2$),
\[
\|v-v_n\|_{H^s(\Omega)}\lesssim\|\Phi-\Phi_n\|_{H^s(G_{R_1})}\le\|\Phi-\Phi_n\|_{H^s(\mathbb{S}^d)}\lesssim\delta^{r-s}\|v\|_{H^r(\Omega)},
\]
with the same coefficients $\mathbf c$ for every $s$.
\end{proof}

\bibliographystyle{siamplain}
\bibliography{references}

\end{document}